%% file: LLMSS_final.tex
\documentclass[12pt]{amsart}
\usepackage{amssymb,latexsym,amsmath,amsthm,amsfonts, enumerate}
\usepackage{color}
\usepackage{graphicx}
\usepackage[all]{xy}
\usepackage{tabu}

\advance \topmargin by -\headheight
\advance \topmargin by -\headsep
\evensidemargin \oddsidemargin
\numberwithin{equation}{section}

\newtheorem{theorem}{Theorem}[section]
\newtheorem{proposition}[theorem]{Proposition}

\newtheorem{corollary}[theorem]{Corollary}
\newtheorem{lemma}[theorem]{Lemma}

\theoremstyle{definition}
\newtheorem{remark}[theorem]{Remark}

\let\oldmarginpar\marginpar
\renewcommand\marginpar[1]{\-\oldmarginpar[\raggedleft\small\sf
#1]{\raggedright\small\sf #1}}

\newcommand{\Hom}{\mathrm{Hom}}
\newcommand{\Aut}{\mathrm{Aut}}
\newcommand{\Ext}{\mathrm{Ext}}

\newcommand{\bk}{\textup{Brac}_{k}(L)}

\newcommand{\dimv}{{\mathbf{dim}}\,}

\newcommand{\za}{\alpha}

\newcommand{\zd}{\delta}

\newcommand{\zg}{\gamma}

\newcommand{\zl}{\lambda}

\newcommand{\cald}{\mathcal{D}}

\begin{document}

\title[Frieze varieties]{Frieze varieties : A characterization of the finite-tame-wild trichotomy for acyclic quivers}

\author{Kyungyong Lee}
\address{Department of Mathematics, University of Nebraska--Lincoln, Lincoln, NE 68588, U.S.A.,
and Korea Institute for Advanced Study, Seoul 02455, Republic of Korea}
\email{klee24@unl.edu; klee1@kias.re.kr}
%\address{Department of Mathematics, 
%University of Nebraska-Lincoln,
%Lincoln NE 68588-0130,
%USA}
%\email{klee24@unl.edu}
\author{Li Li}
\address{Department of Mathematics
and Statistics,
Oakland University, 
Rochester, MI 48309-4479, USA 
}\email{li2345@oakland.edu}
\author{Matthew Mills}
\address{Department of Mathematics, 
University of Nebraska-Lincoln,
Lincoln NE 68588-0130,
USA}
\email{matthew.mills@huskers.unl.edu}

\author{Ralf Schiffler}\thanks{The first author was supported by the University of Nebraska--Lincoln, Korea Institute for Advanced Study, and NSA grant H98230-16-1-0059. 
The second author was supported by NSA grant H98230-16-1-0303. The third author was supported by the University of Nebraska.
The fourth author was supported by NSF-CAREER grant  DMS-1254567, and by the University of Connecticut. The fifth author was supported by NSF grant DMSÐ1601024.}
\address{Department of Mathematics, University of Connecticut, 
Storrs, CT 06269-3009, USA}
\email{schiffler@math.uconn.edu}

\author{Alexandra Seceleanu}
\address{Department of Mathematics, 
University of Nebraska-Lincoln,
Lincoln NE 68588-0130,
USA}
\email{aseceleanu@unl.edu}

\subjclass[2010]{Primary 16G60 %Representation type (finite, tame, wild, etc.
Secondary 13F60 %cluster algebras
16G20 %Representations of quivers and partially ordered sets
14M99 %special varieties
} 
\maketitle
\setcounter{tocdepth}{1}
\tableofcontents
\begin{abstract} We introduce a new class of algebraic varieties which we call frieze varieties. Each frieze variety is determined by an 
acyclic quiver.   The frieze variety is defined in an elementary recursive way by constructing a set of points in affine space. From a more conceptual viewpoint, the coordinates of these points are specializations of cluster variables in the cluster algebra associated to the quiver.  

We give a new characterization of the finite--tame--wild trichotomy for acyclic quivers  in terms of their frieze varieties.
 We show that an acyclic quiver is representation finite, tame, or wild, respectively, if and only if the dimension of its frieze variety is $0,1$, or $\ge2$, respectively. 
 \end{abstract}

\section{Introduction}

For every acyclic quiver $Q$, we define an algebraic variety $X(Q)$ which we call the \emph{frieze variety} of $Q$. The terminology stems from the fact that for quivers of Dynkin type $\mathbb{A}$ the coordinates of the points of the frieze variety are entries in  Conway-Coxeter friezes \cite{CC}. 
The frieze variety gives a geometric interpretation of the quiver as well as concrete  
numerical invariants, for example the dimension, the number of components and the degree. 

The construction of the variety $X(Q)$ is inspired from the theory of cluster algebras. It is defined as follows. 
Let $Q=(Q_0,Q_1)$ be an acyclic quiver  (i.e., a directed graph without oriented cycles) with $n$ vertices. Then we can label the vertices by integers $1,\dots,n$ such that  $i>j$ if there is an arrow $i\to j$.

For every vertex $i\in Q_0$ we define positive rational numbers $f_i(t)$ ($t\in\mathbb{Z}_{\ge0}$) recursively by  $f_i(0)=1$ and 
\begin{equation}\label{def:fi}
f_i(t+1)=\frac{1+\prod_{j\to i}f_j(t) \prod_{j\leftarrow i}f_j(t+1)}{f_i(t)}.
\end{equation}
We will see in Lemma \ref{lem 123}  below that these $f_i(t)$ are exactly the specializations at $x_1=\cdots=x_n=1$ of preprojective cluster variables in the cluster algebra of $Q$. In particular the $f_i(t)$ are  integers.

For every $t$, we thus obtain a point  $P_t=(f_1(t),\dots,f_n(t))\in\mathbb{C}^n$  in an affine space.
We define the \emph{frieze variety} $X(Q)$ of the quiver $Q$ to be the Zariski closure of the set of all  points $P_t$ ($t\in\mathbb{Z}_{\ge0}$). 
If we choose a different labelling, then the coordinates of each new $P_t$ are obtained from the old by permuting in the same way for every $t$. So the new $X(Q)$ is obtained from the old by permuting its coordinates, thus is isomorphic to the old one. In particular, $\dim X(Q)$ is independent of the labelling.

\smallskip

Thus every acyclic quiver $Q$ comes with an algebraic variety $X(Q)$. At this point many natural questions arise. Does the geometry of the variety reflect the representation theory of the quiver? Which quivers have smooth frieze varieties? Are dimension, degree and number of components meaningful invariants of the quiver? Moreover, 
although we focus in this paper on acyclic quivers, one can easily generalize the definition of a frieze variety to quivers that are not acyclic themselves but that are mutation equivalent to an acyclic quiver. One may then ask how does the frieze variety behave under mutation? 
\smallskip

In this paper, we show that the dimension of the frieze variety detects the representation type of the quiver.
An acyclic quiver $Q$ is either {\em representation finite}, {\em tame} or {\em wild}, depending on the representation theory of its path algebra. The quiver is representation finite  if and only if its underlying graph is a Dynkin diagram of type $\mathbb{A,D} $ or $\mathbb{E}$ \cite{G}, and it is tame if and only if the underlying graph is an affine Dynkin diagram of type $\widetilde{\mathbb{A}},\widetilde{\mathbb{D}}$ or $\widetilde{\mathbb{E}}$ \cite{DF,N,DR}.
All other acyclic quivers are wild.

 We propose a new characterization of the finite--tame--wild trichotomy in terms of the frieze variety $X(Q)$ of the quiver $Q$.
\begin{theorem}\label{thm main}
Let $Q$ be an acyclic quiver.
\begin{itemize}
\item [{\rm (a)}] If $Q$ is representation finite then  the frieze variety $X(Q)$ is of dimension 0.
\item [{\rm (b)}] If $Q$ is tame then the frieze variety $X(Q)$ is of dimension 1.
\item [{\rm (c)}] If $Q$ is wild then the frieze variety $X(Q)$ is of dimension at least 2.
\end{itemize}
\end{theorem}

If $Q$ is representation finite then the cluster algebra has only finitely many cluster variables \cite{FZ2} and hence $X(Q)$ is a finite set of points. This shows   part (a) of Theorem \ref{thm main}.

To prove part (b), we will specify linear recursions for the coordinates of the points $P_t$ and then use a general argument to show that the projection of $X(Q)$ to any coordinate plane is contained in the zero locus of a polynomial constructed from the linear recurrence. The key step here is to show that all the roots of the characteristic polynomials of all recursions are integral powers of a single complex number.
Linear recursions for the sequences $(f_i(t))_{t\ge 0}$ where  already considered in \cite{ARS, KS}, where it is shown that there exists a linear recursion for $(f_i(t))_{t\ge 0}$ for all $i$ if and only if $Q$ is representation finite or tame. In  \cite{KS}, explicit linear recursions were given in type  $\widetilde{\mathbb{D}}$ for leaf vertices, and we give new proofs for these recursions here. For type $\widetilde{\mathbb{A}}$ as well as the  non-leaf vertices in type  $\widetilde{\mathbb{D}}$, we provide new  explicit recursions.

To prove part (c) of the theorem, we use the fact that the points $P_t$ correspond to slices  $\tau^{-t+1} kQ$ in the preprojective component of the Auslander-Reiten quiver of the path algebra $kQ$ of $Q$, as well as several known facts on the spectral theory of the Coxeter matrix of a wild quiver, see \cite{delaPena}. 
%we first show that the point $P_t$ corresponds to the cluster determined by $\tau^{-t+1} kQ$ where $kQ$ denotes the path algebra of $Q$ as a module over itself, and $\tau^{-1}$ the inverse Auslander-Reiten translation. The coordinates $f_i(t)$ are the number of terms in the Laurent expansion of the cluster variable corresponding to the representation $\tau^{-t+1}P(i)$, where $P(i)$ denotes the indecomposable projective representation at vertex $i$.
%
The key result, which we think  interesting in its own right, is to show that, when $t$ goes to infinity, the natural logarithm  of the coordinates $\ln f_i(t)$ grows in the same way as $\rho^t$, where $\rho$ is the  largest eigenvalue, or spectral radius, of the Coxeter matrix. See Proposition \ref{limit L(t)}.

There are several characterizations of the finite-tame-wild trichotomy. In \cite{R}, Ringel showed that $Q$ is wild if and only if  the spectral radius of the Coxeter transformation is greater than 1. In \cite{SW}, Skowro\'nski and Weyman characterized tameness in terms of semi-invariants.
Recently, Lorscheid and Weist  characterized tameness using quiver Grassmannians \cite{LW}. 
To our knowledge, our characterization is the first one in terms of numerical invariants that are integers.
%They showed that $Q$ is representation finite if and only if for every indecomposable representations $M$ and all dimension vectors $e$ the quiver Grassmannians $\textup{Gr}_e(M)$  is smooth and have a cell decomposition into affine spaces;
% $Q$ is tame if and only if  all quiver Grassmannians of indecomposable representations of $Q$ have a cell decomposition into affine spaces, but there exist an indecomposable representation $M$ and a dimension vector $e$ such that $\textup{Gr}_e(M)$ is singular; 
%$ Q$ is wild if and only if every integer can be realized as the Euler characteristic of a quiver Grassmannian for $Q$.

The paper is organized as follows. In Section \ref{sect 2}, we recall several definitions and results from representation theory and cluster algebras that are needed later. We prove part (b) of Theorem \ref{thm main} in Section \ref{sect tame} and part (c) in Section \ref{sect wild}. We give several examples in section \ref{sect 4}.

%%%%%%%%%%%%%%%%%%%%%%%%%%%%%%%%%%%%%%%%%%%%%%%%%%%%%%%%%%%%
%
%
%   PRELIMINARIES
%
%
%%%%%%%%%%%%%%%%%%
\section{Preliminaries}\label{sect 2}
Throughout the paper we work over the field of complex numbers $\mathbb{C}$.
\subsection{Quivers and representations}\label{subsection: basic definition}
We start by recalling a few basic facts about quivers and their representations. For further details we refer to \cite{ASS, Schiffler}.

A \emph{quiver} $Q=(Q_0,Q_1,s,t)$ consists of a set $Q_0$ of vertices, a set $Q_1$ of arrows, and two maps $s,t: Q_1\to Q_0$ that send an arrow $\alpha$ to its starting point $s(\alpha)$ and terminal point $t(\alpha)$. We call $Q$ a \emph{finite quiver} if $Q_0$ and $Q_1$ are both finite sets. We will always assume $Q$ to be finite in our paper. 

 A \emph{representation} $M=(M_i,\varphi_\alpha)_{i\in Q_0, \alpha\in Q_1}$ of a quiver $Q$ is a collection of $\mathbb{C}$-vector spaces $M_i$ ($i\in Q_0$) together with a collection of $\mathbb{C}$-linear maps $\varphi_\alpha: M_{s(\alpha)}\to M_{t(\alpha)}$ ($\alpha \in Q_1$).   A representation $M$ is called \emph{finite-dimensional} if each $M_i$ is finite-dimensional. The representations considered in this paper are all finite-dimensional.   Let  $\dimv M=(\dim M_i)_{i\in Q_0}$ be the \emph{dimension vector} of $M$, and let $\textup{rep}_{\mathbb{C}}Q$ denote the category of finite-dimesional representations of $Q$. 
Let $\mathbb{C}Q$ be the path algebra of the quiver $Q$ over $\mathbb{C}$ and  $\textup{mod}\,\mathbb{C}Q$ the category of finitely generated $\mathbb{C}Q $-modules. There is an equivalence of categories $\textup{mod}\,\mathbb{C}Q\cong \textup{rep}_\mathbb{C}(Q)$, and we use the notions of representations and modules interchangeably.

%A morphism between two representations $M=(M_i,\varphi_\alpha), M'=(M'_i,\varphi_\alpha')$ of $Q$ is a collection of linear maps $(f_i: M_i\to M_i')_{i\in Q_0}$ such that $f_j\circ \varphi_\alpha=\varphi'_\alpha\circ f_i$ for any arrow $i\stackrel{\alpha}{\to} j$ in $Q_1$. 

The \emph{projective representation} $P(i)$ at vertex $i\in Q_0$ is defined as $(P(i)_j, \varphi_\alpha)$ where $P(i)_j$ is the vector space with basis the set of all paths from $i$ to $j$ in $Q$; for an arrow $j\stackrel{\alpha}{\to}\ell$ in $Q$, the map $\varphi_\alpha: P(i)_j\to P(i)_\ell$ is determined by composing the paths from $i$ to $j$ with the arrow  $j\stackrel{\alpha}{\to}\ell$. 

Let $D$ be the duality functor $\Hom_\mathbb{C}(-,\mathbb{C})$, and let $A=\oplus_{j\in Q_0} P(j)$. The \emph{Nakayama functor} is defined as $\nu=D\Hom_A(-,A)$. 
Let $\tau,\tau^{-1}$ denote the Auslander-Reiten translations. 
Recall the  definition of $\tau$. 
Let  $M$ be an indecomposable, non-projective representation of an acyclic quiver $Q$, and
\[\xymatrix{P_1\ar[r]^f&P_0\ar[r]&M\ar[r]&0}\] be a minimal projective presentation. Then $\tau M$ is defined by the following exact sequence 
\[\xymatrix{0\ar[r]& \tau M\ar[r]  &\nu P_1\ar[r]^{\nu f}& \nu P_0\ar[r]& \nu M\ar[r]& 0}.\]
 The inverse Auslander-Reiten translation $\tau^{-1}$ is defined dually for non-injective, indecomposable representations of $Q$.
 For every indecomposable non-injective representation $M$ there is a unique almost split sequence $0\to M\to E \to \tau^{-1}M\to 0$ starting at $M$.
 
An indecomposable representation $N$ is called \emph{preprojective} if there is a nonnegative integer $t$ such that $\tau^{t}N=P(i)$ for some $i\in Q_0$. The set of all indecomposable preprojective representations form the \emph{preprojective component} of the Auslander-Reiten quiver of $Q$.

\subsubsection{Admissible sequences} A sequence of vertices $(i_1,\dots,i_n)$ ($i_j\neq i_\ell$ if $j\neq \ell$) is called an \emph{admissible sequence} if the following conditions hold:

(1) $i_1$ is a sink of $Q$;

(2) $i_2$ is a sink of the quiver $s_{i_1}Q$ obtained from $Q$ by reversing all arrows that are incident to the vertex $i_1$;

(3) $i_t$ is a sink of $s_{i_{t-1}}\cdots s_{i_1}Q$ for $t=2,3,\dots,n$.
\smallskip

Note that the above definition is equivalent to saying that $j<\ell$ if there is an arrow $i_j\leftarrow i_\ell$ in $Q$. Indeed, assuming $(i_1,\dots,i_n)$ is admissible, if there is an arrow $\alpha: i_j\leftarrow i_\ell$ in $Q$, then  the sequence $s_{i_{\ell-1}}\cdots s_{i_1}$ changes the orientation of
$\alpha$ if and only if exactly one of $i_j$ and $i_\ell$ is in $\{i_1,\dots,i_{\ell-1}\}$, or equivalently, $j<\ell$ (because $i_\ell\notin \{i_1,\dots,i_{\ell-1}\}$). Conversely, assume $j<\ell$ if there is an arrow $i_j\leftarrow i_\ell$ in $Q$.
Then any arrow of the form $i_j\leftarrow i_t$ (thus $j<t$),  changes its orientation under the sequence $s_{i_{t-1}}\cdots s_{i_1}$, so we get a new arrow $i_j\to i_t$. On the other hand,  any arrow of the form $i_t\leftarrow i_j$ (thus $t<j$), remains unchanged under the sequence $s_{i_{t-1}}\cdots s_{i_1}$. Thus $i_t$ becomes a sink in the quiver $s_{i_{t-1}}\cdots s_{i_1}Q$. 

It is easy to see that if $(i_1,\dots,i_n)$ is an admissible sequence then  $s_{i_{n}}\cdots s_{i_1}Q=Q$. Indeed, since the admissible sequence contains each vertex exactly once, the reflection sequence reflects each arrow exactly twice.

Since we chose our vertex labels $1,\ldots,n$ such that $i>j$ if there is an arrow $i\to j$, we see that the sequence $1,\ldots,n$ is an admissible sequence.
\subsubsection{Structure of the preprojective component of the Auslander-Reiten quiver}\label{sect AR} The preprojective component of the Auslander-Reiten quiver has vertices $\tau^{-t} P(i)$ with $i \in Q_0$, $t\ge 0$ and arrows  $\alpha_t\colon\tau^{-t} P(i)\to  \tau^{-t} P(j)$ and 
 $\overline{\alpha}_t\colon\tau^{-t} P(j)\to  \tau^{-t-1} P(i)$, for all $\alpha\colon j\to i \in Q_1$, $t\ge 0$. For example,  if $Q$ is the quiver $\xymatrix{1&2\ar[l]_\alpha&3\ar[l]<2pt>^\beta\ar[l]<-2pt>_\gamma}$, then the beginning of the preprojective component is of the form
 \[\scriptsize\xymatrix@R15pt@C15pt{
 &&P(3) 
 \ar[rrd]<2pt>^{\overline{\beta}_0}
 \ar[rrd]<-2pt>_{\overline{\gamma}_0}  
 &&&\tau^{-1}P(3)
 \ar[rrd]<2pt>^{\overline{\beta}_1}
 \ar[rrd]<-2pt>_{\overline{\gamma}_1}  
 &&&\tau^{-2}P(3)\\
 &P(2)\ar[ru]<2pt>^{\beta_0}\ar[ru]<-2pt>_{\gamma_0} 
 \ar[rrd]^{\overline{\alpha}_0}
 &&&\tau^{-1}P(2)\ar[ru]<2pt>^{\beta_1}\ar[ru]<-2pt>_{\gamma_1}
  \ar[rrd]^{\overline{\alpha}_1}
 &&&\tau^{-2}P(2)\ar[ru]<2pt>^{\beta_2}\ar[ru]<-2pt>_{\gamma_2}
 \\
 P(1)\ar[ru]^{\alpha_0}
 &&&\tau^{-1} P(1)\ar[ru]^{\alpha_1}
 &&&\tau^{-2} P(1)\ar[ru]^{\alpha_2}
 &&\cdots
} \]
Each mesh of the preprojective component represents an almost split short exact sequence of the following form
\[\xymatrix{0\ar[r] &\tau^{-t+1}P(i) \ar[r] &\displaystyle \bigoplus_{j\to i} \tau^{-t+1} P(j) \oplus\bigoplus_{i\to j} \tau^{-t}P(j)\ar[r] &\tau^{-t}P(i)\ar[r]& 0}\]
\subsection{Cluster algebras} Let $Q$ be an acyclic quiver with $n$ vertices.
The cluster algebra $\mathcal{A}(Q)$ of the quiver $Q$ is the $\mathbb{Z}$-subalgebra of the field of rational functions $\mathbb{Q}(x_1,\ldots,x_n)$ generated by the set of all cluster variables obtained by mutation from the initial seed $((x_1,\ldots,x_n),Q)$.  
For every vertex $i$, the mutation $\mu_i $ in direction $i$ transforms a seed $((x_1,\ldots,x_n),Q)$ by replacing the $i$-th cluster variable $x_i$ by the new cluster variable $({\prod_{j\to i} x_j +\prod_{ i\to j} x_j})/{x_i}$,
where the first product runs over all arrows in $Q$ that end at $i$ and the second product over all arrows that start at $i$. Moreover, the mutation also changes the quiver. 
We refer to \cite{FZ4} for further details on cluster algebras.

In this paper, we are only concerned with mutations at sinks. Recall that a vertex $i$ is a sink if there is no arrow starting at $i$. Thus in this case, the mutation formula becomes  $(\prod_{j\to i} x_j +1)/{x_i}$. Moreover on the level of the quiver, the mutation $\mu_i Q$ of $Q$ at a sink $i$ is the same as the reflection $s_i Q$ of $Q$. 

Let $i_1,\ldots,i_n$ be an admissible sequence for $Q$, and denote the corresponding mutation sequence $\boldsymbol{\mu}=\mu_{i_n}\cdots\mu_{i_1}$.
Then each mutation in this sequence is a mutation at a sink, and moreover $\boldsymbol{\mu} Q =Q$.
Let $\mathbf{x}_0=(x_1(0),\ldots,x_n(0))$ denote the initial cluster and $\mathbf{x}_t=(x_1(t),\ldots,x_n(t))=\boldsymbol{\mu}^t (\mathbf{x}_0)$ be the cluster obtained from it by applying the sequence $\boldsymbol{\mu}$ exactly $t$ times, where $x_j(t)$ is the unique cluster variable that appears for the first time after the mutations 
$\mu_{i_j}\cdots\mu_{i_1} \boldsymbol{\mu}^{t-1}$.

A representation $M$ is called \emph{rigid} if $\Ext^1(M,M)=0$. For example, indecomposable preprojective representations are rigid, since $\Ext^1(\tau^{-t}P(i),\tau^{-t}P(i))\cong\Ext^1(P(i),P(i))=0$.
The {\em cluster character}, or \emph{Caldero-Chapoton map}, associates a  cluster variable $X_M$ to every indecomposable, rigid representation $M$ of $Q$ in such a way that the denominator of Laurent polynomial $X_M$ is equal    to $\prod_{i\in Q_0} x_i^{d_i}$, where $(d_1,\ldots,d_n)$ is the dimension vector of $M$. This was shown in  \cite{CK} for Dynkin quivers and in \cite{CK2} for arbitrary acyclic quivers. 
This result applies in particular to all indecomposable preprojective representations $\tau^{-t} P(i)$ with $i\in Q_0, t \ge 0$.

 It was also shown in \cite{BMRRT,CK2} that if $M$ is a rigid indecomposable representation with almost split sequence $0\to M\to E\to\tau^{-1} M\to 0$, then in the cluster algebra we have the exchange relation $X_{\tau^{-1} M} X_M=X_E +1$.
Using our description of the almost split sequences in the preprojective component
 in section \ref{sect AR}, we see  that  
 \[X_{\tau^{-t} P(i)} =
\Big(\prod_{j\to i} X_{\tau^{-t+1} P(j)} \prod_{i\to j} X_{\tau^{-t}P(j)}\ +\ 1\Big)\Big/X_{\tau^{-t+1} P(i),}\]
and with  our notation above this becomes
\[x_i(t+1)= \Big(\prod_{j\to i} x_j(t)  \prod_{i\to j} x_j(t+1)\ +\ 1 \Big)\Big/x_i(t).\]
If we now specialize the initial cluster variables at 1, we obtain precisely the recursive definition of the coordinates $f_i(t)$. We summarize the above results in the following Lemma.
\begin{lemma}\label{lem 123}
\begin{itemize}
 \item[{\rm (1)}] $f_i(t)$ is $x_i(t)$ specialized at $x_1=\cdots=x_n=1$. In particular, $f_i(t)$ is a positive integer.

\item[{\rm (2)}] $x_i(t) = X_{\tau^{-t+1}P(i)}$.

\item[{\rm (3)}] The denominator of $ X_{\tau^{-t+1}P(i)}$ is equal to $\prod_{i=1}^n x_i^{d_i}$, where $(d_1,\ldots,d_n)$ is the dimension vector of ${\tau^{-t+1}P(i)}$
%$\dimv \tau^{-1}M = \Phi^{-1} (\dimv M)$.

\end{itemize}
\end{lemma}

\subsection{Surface type} \label{sect surface}
A special class of quivers are those associated to triangulations of surfaces with marked points. The cluster algebras of these quivers are said to be of surface type. The cluster algebra (with trivial coefficients) does not depend on the choice of triangulation of the surface. It was shown in \cite{FST} that % the cluster variables of a cluster algebra of surface type are in bijection with the tagged arcs in the surface. Moreover, 
there are precisely four types of surfaces that give rise to acyclic quivers.
\begin{enumerate}
\item The disk with $n+3$ marked points on the boundary corresponds to the finite type $\mathbb{A}_n$. The quiver is acyclic if and only if the triangulation has no internal triangles. 
\item The disk with one puncture and $n$ marked points on the boundary corresponds to the finite type $\mathbb{D}_n$. The quiver is acyclic if and only if the triangulation has no internal triangles and exactly two arcs incident to the puncture.
\item The annulus with $p$ marked points on one and $q$ marked points on the other boundary component corresponds to the affine type $\widetilde{\mathbb{A}}_{p,q}$ with $n=p+q$ vertices. The quiver is acyclic if and only if every arc in the triangulation connects two points on different boundary components.
\item The disk with two punctures and $n-3$ marked points on the boundary corresponds to the affine type $\widetilde{\mathbb{D}}$ with $n$ vertices (in the usual notation this would be type $\widetilde{\mathbb{D}}_{n-1}$). The quiver is acyclic if and only if \begin{itemize}
\item [(i)] for each of the two punctures $p_i$ there are precisely two arcs $\tau_{i1}$ and $\tau_{i2}$ $(i=1,2)$ that connect $p_i$ to a boundary point $a_{i1}, a_{i2}$, such that, either $a_{i1}=a_{i2}$ or $a_{i1}$ and $a_{i2}$ are neighbors on the boundary. Therefore, either the arcs $\tau_{i1},\tau_{i2}$ form a selffolded triangle or they form a triangle together with the boundary segment $a_{i1}{\quad\over\quad} a_{i2}$.
\item [(ii)] letting $B_1$ and $B_2$ be the two parts of the boundary separated by the two triangles incident to the punctures, each of the remaining $n-4 $ arcs must connect a point of $B_1$ to a point of $B_2$.
\end{itemize}
\end{enumerate}

It was also shown in \cite{FST} that there is a bijection between cluster variables and tagged arcs in the surface. Later, in \cite{MSW}, combinatorial formulas were given for cluster variables, and in \cite{MSW2} these formulas were used to associate elements of the cluster algebra to other curves in the surface including closed simple loops and bracelets. If $L$ is a closed simple curve its $k$-bracelet $\bk$ is the $k$-fold concatenation of $L$ with itself. Thus the 1-bracelet is just the loop $L$ and the $k$-bracelet has $k-1$ selfcrossings. These bracelets are essential in the construction of the canonical basis known as the bracelet basis in \cite{MSW2}. Bracelets satisfy the following Chebyshev recursion
$\textup{Brac}_{0}(L)=2, \textup{Brac}_{1}(L)=L$ and \[ \bk = L\cdot \textup{Brac}_{k-1}(L) - \textup{Brac}_{k-2}(L). \]
All these elements satisfy the so-called skein relations, which are given on the level of curves by smoothing a crossing $\times$ in two ways   $\genfrac{}{}{0pt}{5pt}{\displaystyle\smile}{\displaystyle\frown}$ and $\supset\subset$. The skein relations in the cluster algebra were proved in \cite{MW} using hyperbolic geometry and in \cite{CS,CS2,CS3} using only the combinatorial definition of the cluster algebra elements. The skein relations between bracelets and arcs play  a crucial role in the proof of our main theorem in the affine types $\widetilde{\mathbb{A}}$ and $\widetilde{\mathbb{D}}$.

\subsection{Linear recurrences}\label{sect lr} We recall the following result about linear recurrences. %In this subsection, let $K$ be an arbitrary field.
\begin{lemma}
\label{lem lr} 
 Let $(a_n)$ be a sequence  given by the recurrence 
$a_n=c_1a_{n-1} +c_2a_{n-2}+\ldots+c_d a_{n-d}$, where the $c_i\in \mathbb{C}$ are constant. Let $p(x)$ be the characteristic polynomial of this recurrence, thus
$p(x)=x^d-c_1x^{d-1} -c_2x^{d-2} -\ldots-c_d,$  and denote by $r_1,r_2,\ldots, r_d$ the roots of $p(t)$.
 If the roots  of $p(t)$ are all distinct then there exist constants $\za_i\in \mathbb{C}$ such that
$a_n=\za_1r_1^n+\za_2r_2^n+\cdots \za_dr_d^n$.

In particular, if there exists a complex number $\rho$ such that the roots are $\rho^i$, for $d$ distinct integers $i$, then 
$a_n=\sum_i \za_i\rho^{ni}$.
\end{lemma}

%%%%%%%%%%%%%%%%%%%%%%%%%%%%%%%%%%%%%%%%%%%%%%%%%%%%%%%%%%%%%%
%
% 
%       TAME
%
%
 %%
%
%
%       TAME 
%
%
 %%
%
% 
%       TAME
%
%
 %%

\section{Proof of the main theorem part (\textup{b}), the tame case}\label{sect tame}
In this section, we prove that $\dim X(Q)=1$ for affine types. In each case we shall  exhibit a linear recursion for the coordinates of the points of $X(Q)$, which then will imply the desired result via the results in section \ref{sect lr}. In the types $\widetilde{\mathbb{A}}$ and $\widetilde{\mathbb{D}}$, we  use skein relations to obtain the recursions, and in type $\widetilde{\mathbb{E}}$ we checked the recursion formulas by computer.

We start with two preparatory lemmas. 

\begin{lemma}
 \label{cor lr}  
Let $e$ be a positive integer and  $\rho\in \mathbb{C}$. Let $(a_n)_{n\in\mathbb{Z}}$ and $(b_n)_{n\in\mathbb{Z}}$ be two sequences such that
$$
a_n= \sum_{i=-e}^e \alpha_i\rho^{ni} \quad \text{ and } \quad b_n= \sum_{i=-e}^e \beta_i\rho^{ni}
$$ with $\alpha_i,\beta_i\in \mathbb{C}$. Then there exists a nonzero polynomial $g(x,y)\in \mathbb{C}[x,y]$ of degree at most $4e-2$ such that $g(a_n,b_n)=0$ for all $n$.

\end{lemma}
\begin{proof} 
 Let $d=4e-2$ and consider the general polynomial of degree $d$ 
 \[g(x,y)=\sum_{0\le j+k\le d} c_{j,k}x^jy^k.\]
 This polynomial has $(d+2)(d+1)/2=8e^2-2e$ coefficients $c_{j,k}$. We want to find coefficients $c_{j,k}$ such that $g(a_n,b_n)=0$ for all integers $n$, which is
 \begin{equation}\label{eq lr}\sum_{0\le j+k\le d} c_{j,k}( \sum_{i=-e}^e \alpha_i\rho^{ni} )
 ^j( \sum_{i=-e}^e \beta_i\rho^{ni})^k =0.
 \end{equation}
 If we write the left hand side as a polynomial in $\rho^{\pm n}$ it is of the form
 \[\sum_{\ell=-ed}^{ed} \gamma_\ell (\rho^n)^\ell,\]
 where $\gamma_\ell$ is  a linear function of the $c_{j,k}, 0\le j+k\le d$. Thus in order to show (\ref{eq lr}) it suffices to show that the system of equations
 \[\gamma_\ell((c_{j,k}))=0\]
 has a nontrivial solution, and this is true since the number of equations $2ed+1 = 8e^2-4e+1$ is strictly smaller than the number of variables $8e^2-2e$.
\end{proof}

\begin{lemma}\label{lem dim1}
Let $K$ be any field. 
Fix a nonnegative integer $e$ and any element $\rho\in K$. Let $n$ be a positive integer, and for each $p\in\{1,...,n\}$, let $\{a_j^{(p)}\}_{j\in\mathbb{Z}}$  be a sequence such that
$$
a_j^{(p)}= \sum_{i=-e}^e \alpha_i^{(p)}\rho^{ji},
$$ where $\alpha_i^{(p)}\in K$ is independent of $j$. Then the Zariski closure of $\{(a_j^{(1)}, a_j^{(2)},\cdots,a_j^{(n)})\in \mathbb{A}_K^n \ : \  j\in\mathbb{Z} \}$ is of Krull dimension $\leq 1$.
\end{lemma}
\begin{proof}
We use induction on $n$. There is nothing to show for $n=1$. The base case of $n=2$ is proved in Lemma~\ref{cor lr}. Suppose that it holds for $n$, and we prove for $n+1$. Let 
$$\aligned 
C&:=\overline{\{(a_j^{(1)}, a_j^{(2)},\cdots,a_j^{(n-1)},a_j^{(n)},0)\in \mathbb{A}_K^{n+1} \ : \  j\in\mathbb{Z} \}},\\
D&:=\overline{\{(a_j^{(1)}, a_j^{(2)},\cdots,a_j^{(n-1)},0,a_j^{(n+1)})\in \mathbb{A}_K^{n+1} \ : \  j\in\mathbb{Z} \}},\\
E&:=\overline{\{(a_j^{(1)}, \cdots, a_j^{(n-2)},0,a_j^{(n)},a_j^{(n+1)})\in \mathbb{A}_K^{n+1} \ : \  j\in\mathbb{Z} \}},\\
Z&:=\overline{\{(a_j^{(1)}, a_j^{(2)},\cdots,a_j^{(n-1)},a_j^{(n)},a_j^{(n+1)})\in \mathbb{A}_K^{n+1} \ : \  j\in\mathbb{Z} \}},
\endaligned$$ where the bar denotes the Zariski closure. By induction each of $\dim(C)$, $\dim(D)$, and $\dim(E)$ are $\leq 1$. If one of them is equal to 0, then $\dim(Z)\leq 1$ since $Z\subset C\times \mathbb{A}^1$, $Z\subset D\times \mathbb{A}^1$, and $Z\subset E\times \mathbb{A}^1$. Suppose that $\dim(C)=\dim(D)=\dim(E)=1$. Aiming at contradiction, assume that $\dim(Z)=2$. Let $Z_1$ be an irreducible component of $Z$ with $\dim(Z_1)=2$. Then $Z_1= C_1\times \mathbb{A}^1 = D_1\times \mathbb{A}^1=E_1\times \mathbb{A}^1$ for some irreducible component $C_1$ of $C$, some irreducible component $D_1$ of $D$, and some irreducible component $E_1$ of $E$. This implies that all of $C_1,D_1$, and $E_1$ are lines. Hence $Z_1$ is a linear plane in $\mathbb{A}_K^{n+1}$. Choose three non-collinear points in general position, say $(q_{w,1},...,q_{w,n+1})_{w\in\{1,2,3\}}$, on $Z_1$. Then the points $(q_{w,1},...,q_{w,n},0)_{w\in\{1,2,3\}}$ are distinct and collinear, because they are on the line $C_1$. Similarly $(q_{w,1},...,q_{w,n-1},0,q_{w,n+1})_{w\in\{1,2,3\}}$ are distinct and collinear, and $(q_{w,1},...,0,q_{w,n},q_{w,n+1})_{w\in\{1,2,3\}}$ are collinear as well. Then  $(q_{w,1},...,q_{w,n+1})_{w\in\{1,2,3\}}$ become collinear, which is a contradiction.   
\end{proof}

%%
%
% Section A
%
%%
\subsection{Affine type $A$}\label{sect A}
Let $Q$ be an acyclic quiver of type $\widetilde{\mathbb{A}}_{p,q}$. We will use the annulus with $p$ marked points on the inner boundary component and $q$ marked points on the outer boundary component as a model for $\textup{mod}\, \mathbb{C} Q$ as described in section \ref{sect surface}. Let
\[k=\frac{p+q}{\gcd(p,q)} \quad \textup{and}\quad m=\textup{lcm}(p,q).\]
Thus $k=(p+q)m/pq$. Let $L$ be the (isotopy class of the) closed simple curve formed by the equator of the annulus, and consider its $k$-bracelet $\bk$. The crossing number $e(\zg,L)$ between any two isoclasses of curves is defined to be the minimum number of crossings between a curve in the isotopy class of $\zg$ and the isotopy class of $L$. We define the constant
\[C(p,q) = X_{\bk}|_{x_i=1}\]
to be the positive integer obtained from the Laurent polynomial $X_{\bk}$ of the bracelet by specializing the initial cluster variables $x_1=\cdots=x_n=1$. Note that, unless one of $p$ or $q$ is 1, the value of  $C(p,q)$   depends on the orientation of the arrows of $Q$. For $q=1$ we have $C(p,1)=T_{p+1}(p+2)$, where $T_p$ is the $p$-th Chebyshev polynomial with $T_0=2$. So $C(1,1)=7, C(2,1)=52, C(3,1)=527, C(4,1)=6726.$ 
For $p=q=2$, there are two possible values, $C(2,2)=34 $ or 47.
%\smallskip
%\begin{center}
%  \begin{tabular}{ | c | c | c | c | c | c | c | c | c | c | c | c |}
%    \hline
%    p & 1 & 2& 3&4&p &2&3\\ \hline
%    q & 1 & 1&1&1&1&2&2 \\ \hline
%    k & 2 & 3&4&5&p+1&2&5 \\ \hline
%    C(p,q) & 7 & 52&527&6726&$T_{p-1}(p+2)$&34 or 47&30248 \\
%    \hline
%  \end{tabular}
%\end{center}
\medskip

We have the following linear recursion for the coordinates $f_i(t)$ of the points defining $X(Q)$.
\begin{theorem}
 \label{thm A}
 Let $Q$ be of type $\widetilde{\mathbb{A}}_{p,q}$ and $m=\textup{lcm}(p,q)$. Then for all $i\in Q_0$ and all $t\ge m$
 \[f_i(t+m)=C(p,q)f_i(t)-f_i(t-m).\]
\end{theorem}
\begin{proof}
 The indecomposable representations in the preprojective component correspond to arcs that connect points on different boundary components. For each such arc $\zg$, the crossing number $e(\zg,L)$ with $L$ is 1, and the crossing number $e(\zg,\bk)$ with the $k$-bracelet is  $k$. Smoothing one of these crossings we obtain the following skein relation
 
\begin{equation}
 \label{eq A1}
 \bk\cdot \zg =\cald^k(\zg)+\cald^{-k}(\zg),
\end{equation}
where $\cald$ denotes the Dehn twist along $L$. We give an example for $k=2 $ in  Figure \ref{fig Dehn}.
\begin{figure}
\begin{center}
 \scalebox{0.8}{\includegraphics{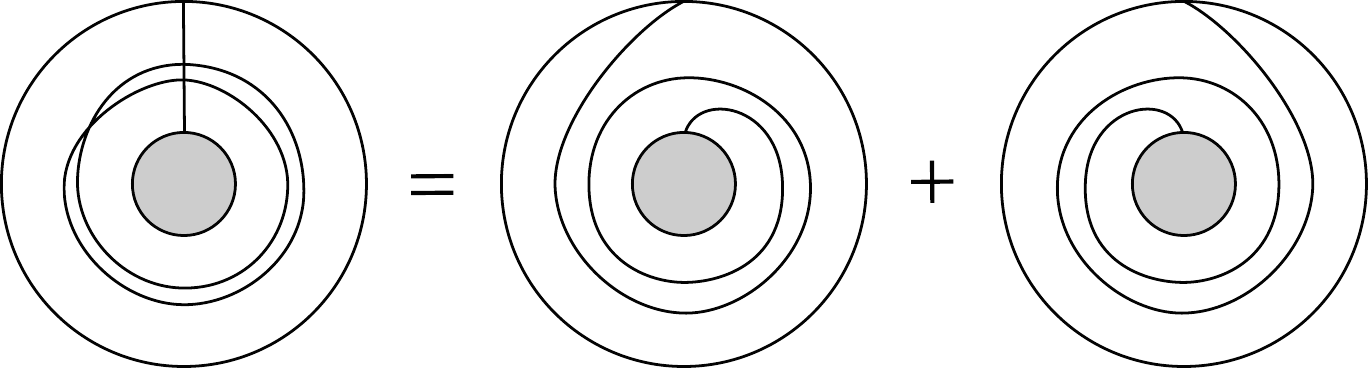}}
\end{center}
\caption{The skein relation of equation (\ref{eq A1}) for $k=2$.}
\label{fig Dehn}
\end{figure}
On the other hand, the inverse Auslander-Reiten translation $\tau^{-1}$ acts on the arc $\zg$ by moving its endpoint on the outer boundary component to its counterclockwise neighbor and its endpoint on the inner boundary component to its clockwise neighbor. Since $m=\textup{lcm}(p,q)$, we see that the arc $\tau^{-m}(g)$ has the same endpoints as $\zg$ but $\tau^{-m}(\zg)$ wraps around the inner boundary component exactly 
$\frac{m}{p}+\frac{m}{q}=\frac{m}{pq}(q+p)=k$ times more than $\zg$. In other words, $\tau^{-m}(\zg)=\cald^k(\zg)$. Therefore, equation (\ref{eq A1}) becomes
\[\bk\cdot \zg =\tau^m(\zg)+\tau^{-m}(\zg),\]
and  passing to the cluster algebra and specializing at $x_i=1$, we have
\[C(p,q) \,f_i(t) = f_i(t-m)+f_i(t+m).\qedhere\]
\end{proof}
\begin{remark}  
 Similar looking recurrence relations involving bracelets but  for arcs that have both endpoints on the same boundary component were found in \cite[Theorem 2.5]{BFPT} and \cite[Theorem 5.4]{GMV}.
\end{remark}
\begin{corollary}
 \label{cor A2}
 Let $Q$ be an acyclic quiver of affine type $\widetilde{\mathbb{A}}$. Then the dimension of the frieze variety $X(Q)$ is equal to one.
\end{corollary}
\begin{proof}
 It suffices to show that, for every pair $i,j\in Q_0$, the projection $\pi_{ij}(X(Q))$ of $X(Q)$ onto the $(i,j)$-plane is of dimension one. By Theorem \ref{thm A} and Lemma \ref{lem lr}, we have linear recurrences 
 \[f_h(t+m)=C(p,q)f_h(t)-f_h(t-m)=%\sum_{\ell=-1}^1\za_{h,\ell} \,C(p,q)^\ell,\]
% \[f_j(t+m)=C(p,q)f_j(t)+f_j(t-m)=\sum_{\ell=-1}^1\za_{j,\ell} \,C(p,q)^\ell,\]
\sum_{e=-1}^{1} \alpha_{h,r,e}\,\rho^{e\lfloor t/m\rfloor},\]
where $\rho$ is one of the roots of the polynomial  $x^2 - C(p,q)x+1$, and $\alpha_{h,r,e}$ depends only on $h\in\{i,j\}, r=t-m\lfloor t/m\rfloor\in\{0,1,...,m-1\} (\textup{mod}\ m)$, and $e$. Since $C(p,q)>2$, we have $\rho\ne 1$.

%with $\za_{h,-1}=0,\za_{h,0}=-f_h(t-m)$ and $\za_{h,1}=f_h(t)$, for $h=i,j$. 
Thus Lemma \ref{cor lr} with $e=1$ implies there is a polynomial $g_t(x,y)$ of degree $4e-2=2$ such that $g_t(f_i(t+sm),f_j(t+sm))=0$, for all $s$ such that $t+sm\ge 0$. Define $g=g_0\cdot g_1\cdots g_{m-1}$. Then $\pi_{ij}(X(Q))$ is contained in the zero locus of $g$, and hence has dimension one.
\end{proof}
%\begin{remark}
%Since the polynomial $g$ is of degree $2m$, this also shows that the minimal polynomial for $X(Q)$ is of degree at most $2m$.
%\end{remark}

%%
%
% Section D
%
%%
\subsection{Affine type $D$}\label{sect D} 
Let $Q$ be an acyclic quiver of type $\widetilde{\mathbb{D}}$ with $n$ vertices. Thus the underlying graph of $Q$ is the following. 
%\[\xymatrix@!@R=-14pt@C=10pt{1\ar@{-}[dr] & &&&& \mbox{$n$\,-$1$}\\
%&3\ar@{-}[r] & 4\ar@{-}[r] & \cdots \ar@{-}[r] & \mbox{$n$\,-$2$}\ar@{-}[ru] \ar@{-}[rd] \\  
%2\ar@{-}[ur]  &&&&& n\\ }\] 
\[\xymatrix@!@R=-4pt@C=10pt{\circ\ar@{-}[dr] & &&&& \circ\\
&\bullet\ar@{-}[r] & \bullet\ar@{-}[r] & \cdots \ar@{-}[r] & \bullet\ar@{-}[ru] \ar@{-}[rd] \\  
\circ\ar@{-}[ur]  &&&&& \circ\\ }\]
Since that we use $n$ for the number of vertices, so in the usual notation this type is $\widetilde{\mathbb{D}}_{n-1}$.
Each of the vertices marked with the symbol $\circ$ is called a {\em leaf} of $Q$ and each of the vertices marked with the symbol $\bullet$ is called a {\em non-leaf}.
We will use the disk with two punctures and $n-3$ marked points on the boundary as a model for $\textup{mod}\,\mathbb{C}Q$ as described in section \ref{sect surface}. Let $L$ be the (isotopy class of the) closed simple curve around the two punctures and let $\bk$ denote its $k$-bracelet.
We denote by $\cald$ the full Dehn  twist along $L$ in counterclockwise direction, and by $\cald^{1/2}$ the half Dehn twist.

The indecomposable representations 
 in the preprojective component correspond to two types of arcs, depending whether the vertex $i$ is a leaf of $Q$ or not.
 \begin{itemize}
\item[$\circ$] If $i$ is a leaf in $Q$,   then $\tau^{-t}P(i)$ corresponds to an arc $\zg$ that connects a boundary point to a puncture such that the crossing number $e(\zg,L)$ is one.
\item[$\bullet$] If $i$ is not a leaf in $Q$, then $\tau^{-t}P(i)$ corresponds to an arc $\zg$ with both endpoints on the boundary and such that the crossing number $e(\zg,L)$ is two.
\end{itemize}
%The preprojective representations of the form $\tau^{-t}P(i)$ are said to lie in the $\tau^{-1}$-orbit of $P(i)$. 
We call an arc $\zg$ {\em preprojective of orbit} $i$ if it corresponds to a preprojective representation of the form $\tau^{-t}P(i)$.

We have the following relation between the Dehn twist and the Auslander-Reiten translation.
\begin{lemma}
 \label{lem D1}
 Let $\zg$ be a preprojective arc of orbit $i$. Then
 \[\cald^2(\zg)=\tau^{-2(n-3)}(\zg).\] Moreover, if $i$ is a non-leaf vertex or $n$ is odd, then also $\cald(\zg)=\tau^{-(n-3)}(\zg)$.
\end{lemma}
\begin{proof}
 Recall that $n-3$ is the number of marked points on the boundary of the disk. If $i$ is a leaf then applying $\tau^{-1}$ moves the endpoint of $\zg $ that lies on the boundary to its counterclockwise neighbor and changes the tagging at the puncture. Thus $\tau^{-(n-3)}(\zg)$ is equal to the full Dehn twist $\cald(\zg)$ if $n-3$ is even, and it is the Dehn twist with opposite tagging at the puncture if $n-3$ is  odd.
 
 If $i$ is not a leaf, then applying $\tau^{-1}$ moves both endpoints of $\zg$ to their counterclockwise neighbors on the boundary. Thus after applying $\tau^{-1}$ exactly $n-3$ times, we have moved each endpoint of $\zg$ counterclockwise around the whole boundary back to its initial position. On the other hand, the Dehn twist does exactly the same, since $\zg$ crosses the loop $L$ twice in this situation.
\end{proof}

If $\zg $ is a preprojective arc of orbit $i$ with $i$ a non-leaf vertex, we let $\zg_1,\zg_2$ be the two arcs that have the same  endpoints as $\zg$ 
and do not cross the loop $L$, see Figure \ref{skein}. Define $S_k(L)$ recursively by $S_1(L)=1$, $S_2(L)=L+2$ and $S_k(L)=LS_{k-1}(L)-S_{k-2}(L)+2$.
\begin{lemma}
 \label{lem D2} Let $\zg$ be a preprojective arc of orbit $i$ and $k\ge 1$. 
 \begin{itemize}
\item [{\rm (a)}] If the vertex $i$ is a non-leaf then 
\[\bk\cdot \zg =\cald^{k/2}(\zg)+\cald^{-k/2}(\zg)+2(\zg_1+\zg_2)\,S_k(L).
\]
\item[{\rm (b)}] If the vertex $i$ is a leaf then
\[\bk\cdot\zg=\cald^k(\zg)+\cald^{-k}(\zg).\]
\end{itemize}
\end{lemma}

\begin{proof}
 For $k=1$ the results are the following skein relations illustrated in Figure \ref{skein}.
\begin{equation}
 \label{eq skein}
 L\cdot \zg =\left\{
\begin{array}
 {ll} 
 \cald^{1/2}(\zg)+\cald^{-1/2}(\zg)+2(\zg_1+\zg_2) &\textup{if $i$ is a non-leaf;}\\
\cald(\zg)+\cald(\zg)&\textup{if $i$ is a leaf.}
 \end{array}
 \right.
\end{equation}
Now suppose $k>1$. In case (a), we may assume by induction that 
\[\textup{Brac}_{k-1}(L) \cdot \zg= \cald^{(k-1)/2}(\zg)+\cald^{-(k-1)/2}(\zg)+2(\zg_1+\zg_2)\,S_{k-1}(L).\]
Multiplying by $L$ and using equation (\ref{eq skein}) we have
\[
\begin{array}
 {rcl}
 L\cdot\textup{Brac}_{k-1}(L) \cdot \zg&=&\cald^{k/2}(\zg)+\cald^{(k-2)/2}(\zg)+ 2(\zg_1+\zg_2)\\[5pt]
 &&+\cald^{-k/2}(\zg)+\cald^{-(k-2)/2}(\zg)+2(\zg_1+\zg_2)+2L(\zg_1+\zg_2)\,S_{k-1}(L)\\[5pt]
 &=& \cald^{k/2}(\zg)+\cald^{-k/2}(\zg)+\textup{Brac}_{k-2}(L)\cdot\zg -2(\zg_1+\zg_2)\,S_{k-2}(L)\\[5pt]
 && + 2(\zg_1+\zg_2)(2+L\cdot S_{k-1}(L)),
\end{array}\]
where the last equation holds by induction. Then
\[ (L\cdot\textup{Brac}_{k-1}(L) -\textup{Brac}_{k-2}(L))\cdot \zg=\cald^{k/2}(\zg)+\cald^{-k/2}(\zg)+2(\zg_1+\zg_2)(L\cdot S_{k-1}(L)-S_{k-2}(L)+2),\]
and using the recursions for the bracelets and for $S_k$ we get
\[\bk\cdot \zg = \cald^{k/2}(\zg)+\cald^{-k/2}(\zg)+2(\zg_1+\zg_2)\,S_k(L).\]

\begin{figure}
 
\begin{center}
\scalebox{0.8}{ 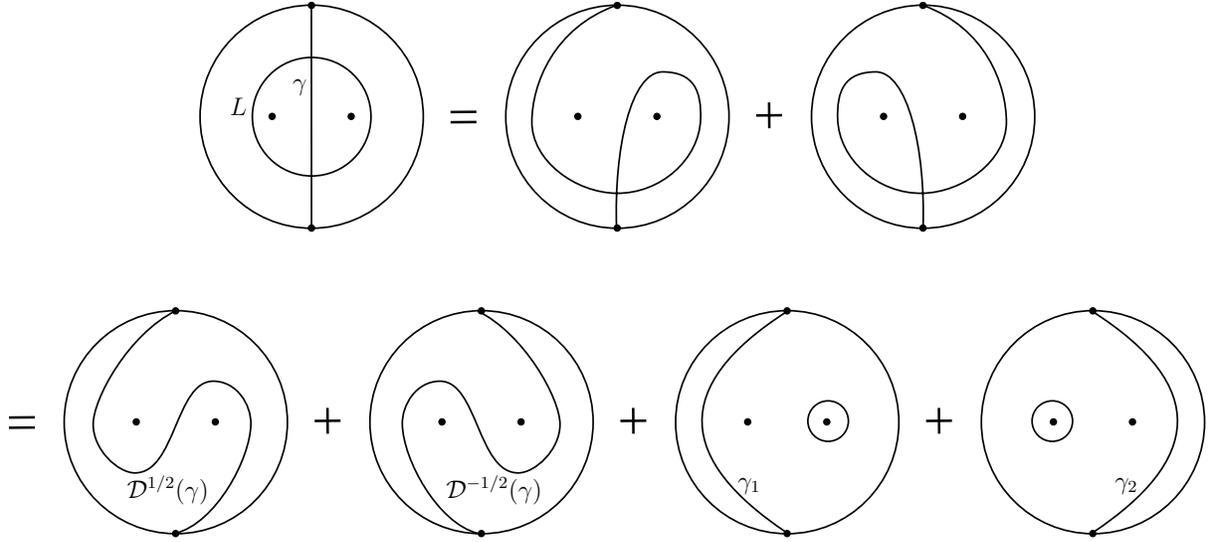}
\end{center}
\caption{The skein relation of equation (\ref{eq skein}). In the cluster algebra the loops that are contractible to the puncture are equal to 2.}\label{skein}
\end{figure}
In case (b), we assume by induction that
\[\textup{Brac}_{k-1}(L) \cdot \zg= \cald^{k-1}(\zg)+\cald^{-(k-1)}(\zg).\]
Again multiplying  by $L$ and using equation (\ref{eq skein}) we have
\[
\begin{array}
 {rcl}
 L\cdot\textup{Brac}_{k-1}(L) \cdot \zg&=&\cald^{k}(\zg)+\cald^{k-2}(\zg)
 +\cald^{-k}(\zg)+\cald^{-(k-2)}(\zg)\\
 &=& \cald^{k}(\zg)+\cald^{-k}(\zg)+\textup{Brac}_{k-2}(L) \cdot \zg
\end{array}\]
where the last equation holds by induction. Then
\[ (
 L\cdot\textup{Brac}_{k-1}(L)-\textup{Brac}_{k-2}(L)) \cdot \zg =  \cald^{k}(\zg)+\cald^{-k}(\zg)
\]
and the left hand side is equal to $\bk\cdot\zg.$
\end{proof} 

For our next result we define the  constants
\[
 C_j(n)=X_{\textup{Brac}_j(L)}|_{x_i=1}\qquad j \in \mathbb Z_{\geq 0}.
\]
Note that $C_2(n)=C_1(n)^2-2$ and the value of the constant depends on the orientation of the edges in $Q$.
We have the following linear recursion for the coordinates $f_i(t)$ of the points defining $X(Q)$.
\begin{theorem}
 \label{thm D}
 Let $Q$ be of type $\widetilde{\mathbb{D}}$ with $n$ vertices. 
 \begin{itemize}
\item [{\rm (a)}] For all non-leaf vertices $i$ and all $t,p\in \mathbb{Z}_{\ge 0}$, we have
 \[f_i(t+3p(n-3))=(C_{2p}(n)+1)\,f_i(t+2p(n-3))-(C_{2p}(n)+1) \,f_i(t+p(n-3))+f_i(t).\]
\item [{\rm (b)}] For all leaf vertices $i$ and all $t\in \mathbb{Z}_{\ge 0}$, we have
\[\begin{array}
{ll} f_i(t+2(n-3))= C_1(n)\, f_i(t+(n-3)) -f_i(t) &\textup{if $n$ is odd;} \\
 f_i(t+4(n-3))= C_2(n) \,f_i(t+2(n-3)) -f_i(t) &\textup{for all $n$.}
\end{array}\]
 \end{itemize}
\end{theorem}
\begin{proof} (a) Let $i$ be a non-leaf vertex.
Combining Lemmas \ref{lem D1} and \ref{lem D2}, we have for every preprojective arc $\zg$ of orbit $i$
\begin{equation}
 \label{eq D2}
 \textup{Brac}_{2p}(L)\cdot \zg = \tau^{-p(n-3)}(\zg)+\tau^{p(n-3)}(\zg)+2(\zg_1+\zg_2)\,S_{2p}(L).
\end{equation}
Let $K$ denote the integer obtained by specializing the Laurent polynomial corresponding to $2(\zg_1+\zg_2)S_{2p}(L)$ at $x_i=1$. Then, if we let $\zg=\tau^{-p(n-3)}\tau^{-t+1}P(i)$, the equation (\ref{eq D2}) yields
\begin{equation}
 \label{eq D3}
C_{2p}(n) \,f_i(t+p(n-3)) = f_i(t+2p(n-3) ) +f_i(t)+K.
\end{equation}
Similarly, if we let  $\zg=\tau^{-2p(n-3)}\tau^{-t+1}P(i)$, the equation (\ref{eq D2}) yields
\begin{equation}
 \label{eq D4}
C_{2p}(n) \,f_i(t+2p(n-3)) = f_i(t+3p(n-3)) +f_i(t+p(n-3))+K.
\end{equation}
Subtracting equation (\ref{eq D3}) from equation (\ref{eq D4}) and rearranging the terms we get
\[f_i(t+3p(n-3)) = (C_{2p}(n)+1) f_i(t+2p(n-3)) -(C_{2p}(n)+1) f_i(t+p(n-3)) +f_i(t).\]
This completes the proof of (a).

(b) Let $i$ be a leaf vertex.
Combining Lemmas \ref{lem D1} and \ref{lem D2}, we have for every preprojective arc $\zg$ of orbit $i$
\[\begin{array}{rcll}
 L\cdot \zg &= &\tau^{-(n-3)}(\zg)+\tau^{n-3}(\zg) & \textup{if $n$ is odd;}\\
 \textup{Brac}_2(L) \cdot \zg &= & \tau^{-2(n-3)}(\zg)+\tau^{2(n-3)}(\zg) & \textup{for all $n$.}
\end{array}\]
We thus obtain
\[\begin{array}{rcll}
 C_1(n)\cdot f_i(t+n-3) &= &f_i(t+2(n-3)) +f_i(t) & \textup{if $n$ is odd;}\\
 C_2(n)\cdot f_i(t+2(n-3)) &= &f_i(t+4(n-3)) +f_i(t) &  \textup{for all $n$.}
\end{array}\qedhere\]
\end{proof}
\begin{remark}
 Part (b) of Theorem \ref{thm D} was obtained in \cite[Theorems 6.1 and 6.2]{KS} using cluster categories.
\end{remark}
\begin{corollary}
 \label{cor D}
 Let $Q$ be an acyclic quiver of affine type $\widetilde{\mathbb{D}}$. Then the dimension of the frieze variety $X(Q)$ is equal to one.
\end{corollary}
\begin{proof} Suppose first that $n$ is odd.
 The characteristic polynomial of the recursions in Theorem \ref{thm D} are
 \[\begin{array}{ll}
 x^3-(C_2(n)+1) x^2+(C_2(n)+1)x-1 &\textup{in case (a) with $p=1$;}\\
 x^2-C_1(n)\,x+1  &\textup{in case (b) with $n$ odd;}\\
% x^2-C_2(n)\,x+1  &\textup{in case (b) for all $n$.}\\
\end{array}
 \]
 
Let $\rho$ be a root of $ x^2-C_1(n)\,x+1 $. % Since the constant term of the polynomial is 1, the other root is $\rho^{-1}$. 
Then $ x^2-C_1(n)\,x+1 = (x-\rho)(x-\rho^{-1})$ and $C_1(n)=\rho+\rho^{-1}$. Moreover $\rho\ne\rho^{-1}$,  since $C_1(n)>2$. From the recursive formula of the bracelets we have
$C_2(n)=C_1(n)^2-2$, thus $C_2(n)+1=\rho^2+\rho^{-2}+1$. Therefore the characteristic polynomial in case (a) is equal to $
 (x-\rho^2)(x-1)(x-\rho^{-2})$.
 In particular the roots of the characteristic polynomials in case (a) and (b) are of the form $\rho^\ell$ with $\ell=-2,-1,0,1,2$. Now the result follows from Lemma~\ref{lem lr} and Lemma~\ref{lem dim1}.
 
 If $n$ is even, we use $p=2$ so that the characteristic polynomial of the recursions in Theorem~\ref{thm D} are
 \[\begin{array}{ll}
 x^3-(C_4(n)+1) x^2+(C_4(n)+1)x-1 &\textup{in case (a) with $p=2$;}\\
% x^2-C_1(n)\,x+1  &\textup{in case (b) with $n$ odd;}\\
 x^2-C_2(n)\,x+1  &\textup{in case (b) for all $n$.}\\
\end{array}
 \]
 Now we let $\rho $ be a root of  $x^2-C_2(n)\,x+1$, use the Chebyshev relation $C_4(n)=C_2(n)^2-2$ and the proof is analogous to the previous case.
 \end{proof}
%%
%
% Section E
%
%%
 \subsection{Affine type $E$}\label{sect E}
 Let $R$ be a commutative ring, and $R'$ be a subring of $R$. 
 If a sequence $(a_\ell)$ of elements in $R$ satisfies a recurrence relation $c_0a_{\ell+k}+c_1a_{\ell+k-1}+\cdots+c_ka_\ell=0$ for all $\ell\ge 0$  (where $c_0,\dots,c_k\in R'$), then we call the polynomial $f(x)=c_0x^k+c_1x^{k-1}+\cdots+c_k\in R'[x]$ an annihilator of $(a_\ell)$.  Recall the following fact: 
 \begin{lemma}\label{lemma:annihilators}
Let $f(x), g(x)\in R'[x]$ be annihilators of $(a_\ell)$ and $(b_\ell)$, respectively. Then:
 
 (i) $f(x)g(x)\in R'[x]$  is an annihilator of $(a_\ell+b_\ell)$ of degree $\deg f+\deg g$. 
 
 (ii) Let $h(x)\in R'[x]$ be the characteristic polynomial of the tensor product of the companion matrices of $f$ and $g$. Then $h(x)$ is an annihilator of $(a_\ell b_\ell)$ of degree $(\deg f)(\deg g)$. 
 
 Moreover, if the leading and constant coefficients of $f, g$ are $\pm 1$, then the same holds for the above annihilators.
  \end{lemma}
\begin{proof}
The parts (i) and (ii) are given in \cite[Lemma 4.1]{KS}. The last statement is obvious.
\end{proof}

Let $Q$ be of type $\tilde{\mathbb{E}}_{n-1}$  with $n$ vertices, so $n=7, 8$, or $9$. 
Let $\delta=(\zd_1,\ldots,\zd_{n})$, $\zd_i>0$ be the unique imaginary Schur root for $Q$; see for example \cite[8.2.1]{Schiffler} for the values of the $\zd_i$. Then there exists a one parameter family $(M_\zl)$ of non-isomorphic indecomposable representations with dimension vector $\zd$ and such that $\textup{End}(M_\zl)\cong k$. Moreover each $M_\zl$ is regular and non-rigid. Choose one such representation $M$ such that $M$ is a regular simple representation. This is equivalent to the condition $\tau M\cong M$, in other words, $M$ sits at the mouth of a homogeneous tube in the Auslander-Reiten quiver. (Note that $X_M$ is denoted $X_\delta$ in \cite{KS}.)

Define
\begin{equation}\label{m}
m_n=\begin{cases}
&6, \textrm{ if }n=7; \\
&12, \textrm{ if } n=8;\\
&30, \textrm{ if } n=9. 
\end{cases}
\quad
\quad\quad
d_n=\begin{cases}
&12, \textrm{ if }n=7; \\
&29, \textrm{ if } n=8;\\
&169, \textrm{ if } n=9. 
\end{cases}
\end{equation}

\begin{lemma}
Let $Q$ be of type $\tilde{\mathbb{E}}_{n-1}$ where $n=7, 8, 9$.   For each $i\in Q_0$, and $0\le r\le m_n-1$, the sequence $\big(f_i(jm_n+r)\big)_{j=0}^\infty$ is annihilated by a polynomial in $\mathbb{Z}[x]$ of degree $d_n$. 
\end{lemma}
\begin{proof}
We explain the conclusion for $n=7$ in detail, since the other two cases are proved similarly. The idea is to give the bounds of degrees of the recursive relations described in \cite{KS}.

First note that the cluster category does not depend on the orientation of $Q$. Once we prove a recurrence relation of $(X_{\tau^{-t}P(i)})_t$ for one particular orientation,  changing the orientation of the quiver will not change the recursion. We use the following orientation of $Q$.
\[\xymatrix{1&2\ar[l]&7\ar[l]\ar[r]\ar[d]&6\ar[r]&5 \\ &&4\ar[d]\\&&3
}
\]
This is the same labelling, and the opposite orientation  as in \cite[\S7]{KS}. (We use the opposite orientation because of a different convention used in \cite{KS}.) 
For simplicity, we denote $m_n$ by $m$. All relations below can be found in  \cite[\S 7]{KS}.

For $i=1, 3, 5$, $\big(X_{\tau^{-(jm+r)}P(i)}\big)_{j=0}^\infty$ is annihilated by $x^2-X_M x+1$. For $i=2$, since
$$X_{\tau^{-t}P(2)}=X_{\tau^{-t}P(1)}X_{\tau^{-(t+1)}P(1)}-1,$$ 
we have
$$X_{\tau^{-(jm+r)}P(2)}=X_{\tau^{-(jm+r)}P(1)}X_{\tau^{-(jm+r+1)}P(1)}-1.$$ Since both  sequences $(X_{\tau^{-(jm+r)}P(1)})_j$ and $(X_{\tau^{-(jm+r+1)}P(1)})_j$ are annihilated by a polynomial of degree $2$, and since the constant sequence $(1)_j$ is annihilated by $x-1$, we conclude that the sequence $(X_{\tau^{-(jm+r)}P(1)}X_{\tau^{-(jm+r+1)}P(1)})_j$ is annihilated by a polynomial of degree $2\cdot 2=4$,  thus $(X_{\tau^{-(jm+r)}P(2)})_j$ is annihilated by a polynomial of degree $4+1=5$, using Lemma \ref{lemma:annihilators}. The same conclusion holds for $i=4, 6$.

For $i=7$, we have
$$X_{\tau^{-(jm+r)}P(7)} = X_{\tau^{-(jm+r)}P(1)} X_{\tau^{-(jm+r+1)}P(2)} - X_{\tau^{-(jm+r+2)}P(1)}.$$

By Lemma \ref{lemma:annihilators}, $(X_{\tau^{-(jm+r)}P(7)})_j$ is annihilated by a polynomial of degree $2\cdot 5 +2 =12$. 

We illustrate these degrees as follows, where the notation $i^{(d)}$ means vertex $i$ corresponding to an annihilating polynomial of degree $d$ with coefficients in $\mathbb{Z}[X_M]$:
\[
\tilde{\mathbb{E}}_6: \quad 
\xymatrix@!@R=+11pt@C=12pt
{1^{(2)}& 2^{(5)}\ar[l]& 7^{(12)}\ar[l]\ar[d]\ar[r]&6^{(5)}\ar[r]&5^{(2)}\\ &&4^{(5)}\ar[d]&&\\    &&3^{(2)}}
\]
Now specialize the annihilating polynomial at $x_1=\cdots=x_n=1$. Since all coefficients are in $\mathbb{Z}[X_M]$, the specialization is well-defined. %For example, we do not need to worry about a coefficient like $1/(X_M-322)$, which after specialization gives $1/0$ if $X_M$ specializes to $322$. 
Moreover, since the leading and constant coefficients are $\pm1$, the specialization is not trivial. This gives us the desired polynomial. 

\smallskip

For $n=8$, the degree of an annihilating polynomial for each vertex are illustrated below:
\[
\tilde{\mathbb{E}}_7: \quad 
\xymatrix@!@R=+11pt@C=12pt
{1^{(2)}& 2^{(5)}\ar[l]&3^{(12)}\ar[l]& 8^{(29)}\ar[l]\ar[d]\ar[r]&6^{(12)}\ar[r]&5^{(5)}\ar[r] &4^{(2)}\\ &&&7^{(5)}&&}
\]
The degrees at vertices $i=1,2,3,8$ are obtained similar as the case $n=7$, using identities 
$$X_MX_{\tau^{-t}P(1)} = X_{\tau^{-t+m}P(1)}+X_{\tau^{-t-m}P(1)},$$
$$X_{\tau^{-t}P(2)} = X_{\tau^{-t}P(1)}X_{\tau^{-(t+1)}P(1)}-1,$$
$$X_{\tau^{-t}P(3)} = X_{\tau^{-t}P(1)}X_{\tau^{-(t+1)}P(2)}-X_{\tau^{-(t+2)}P(1)},$$
$$X_{\tau^{-t}P(8)} = X_{\tau^{-t}P(1)}X_{\tau^{-(t+1)}P(3)}-X_{\tau^{-(t+2)}P(2)}.$$

The degrees of $i=4, 5, 6$ are obtained by symmetry. 
  For the vertex 7, we use the first two exchange triangles in \cite[page 1857]{KS}:
$$P_1\to \tau^{-1}P_7\to \tau^{-4}P_4\to \tau P_1,\quad \tau^{-4}P_4\to N\to P_1\to \tau^{-3}P_4$$
(where $N$ is the indecomposable regular simple module of dimension vector 11100101 %$\tau^{-1} N$ has dimension vector $01100011$ 
 which belongs to the mouth of the tube of width 4) to obtain 
$$X_{\tau^{-(t+1)}P(7)} = X_{\tau^{-t}P(1)}X_{\tau^{-(t+4)}P(4)}-X_{\tau^{-t}N} $$
Note that the sequence $(\tau^{t}N)_t$ has period 4 which divides $m=12$, thus $(X_{\tau^{-(jm+r-1)}N})_j$ is a constant sequence. Substituting $t=jm+r-1$ and using the fact that  $(X_{\tau^{-(jm+r-1)}P(1)})_j$, $(X_{\tau^{-(jm+r+3)}P(4)})_j$, $(X_{\tau^{-(jm+r-1)}N})_j$  are annihilated by polynomials of degrees $2, 2, 1$, respectively, we conclude that the sequence $(X_{\tau^{-(jm+r)}P(7)})_j$ is annihilated by a polynomial of degree $2\cdot 2 +1 = 5$.
\smallskip

For $n=9$, we obtain 
\[
\tilde{\mathbb{E}}_8: \quad 
\xymatrix@!@R=+11pt@C=12pt
{1^{(2)}& 2^{(5)}\ar[l]&3^{(12)}\ar[l]&4^{(29)}\ar[l]&5^{(70)}\ar[l]& 9^{(169)}\ar[l]\ar[r]\ar[d]&8^{(29)}\ar[r]&7^{(5)}\\ &&&&&6^{(12)}}
\]
Indeed, for vertices $i=1,2,3,4,5,9$, we use
$$X_MX_{\tau^{-t}P(1)} = X_{\tau^{-t+m}P(1)}+X_{\tau^{-t-m}P(1)},$$
$$X_{\tau^{-t}P(2)} = X_{\tau^{-t}P(1)}X_{\tau^{-(t+1)}P(1)}-1,$$
$$X_{\tau^{-t}P(3)} = X_{\tau^{-t}P(1)}X_{\tau^{-(t+1)}P(2)}-X_{\tau^{-(t+2)}P(1)},$$
$$X_{\tau^{-t}P(4)} = X_{\tau^{-t}P(1)}X_{\tau^{-(t+1)}P(3)}-X_{\tau^{-(t+2)}P(2)}.$$
$$X_{\tau^{-t}P(5)} = X_{\tau^{-t}P(1)}X_{\tau^{-(t+1)}P(4)}-X_{\tau^{-(t+2)}P(3)}.$$
$$X_{\tau^{-t}P(9)} = X_{\tau^{-t}P(1)}X_{\tau^{-(t+1)}P(5)}-X_{\tau^{-(t+2)}P(4)}.$$

For vertex $7$ we use 
$$X_{\tau^{-t}P(7)} = X_{\tau^{-(t-2)}P(1)}X_{\tau^{-(t+5)}P(1)}-X_{\tau^{-(t+2)}N}$$
(where $N$ is the indecomposable regular simple module of dimension vector $001111001$ which belongs to the mouth of the tube of width 5) and that $5$ divides $m=30$.

\noindent For vertex 6 we use 
$$X_{\tau^{-t}P(6)} = X_{\tau^{-(t-1)}P(1)}X_{\tau^{-(t+2)}P(7)}-X_{\tau^{-(t+7)}P(1)}$$

\noindent For vertex 8 we use 
$$X_{\tau^{-t}P(8)} = X_{\tau^{-(t-1)}P(1)}X_{\tau^{-(t+1)}P(6)}-X_{\tau^{-(t+3)}P(7)}$$
This completes the proof.
\end{proof}

We need to following simple fact.
\begin{lemma}\label{lem:0 sequence}
If a sequence $(a_j)_{j=0}^\infty$ is annihilated by a polynomial of degree $d$, and
$c_0a_{n+k}+c_1a_{n+k-1}+\cdots+c_ka_n=0$ holds for $0\le n\le d-1$, then the equality holds for every $n$.
\end{lemma} 
\begin{proof}
The sequence $(c_0a_{n+k}+c_1a_{n+k-1}+\cdots+c_ka_n)_n$ satisfies a recursive relation of degree $d$ and its first $d$ terms are $0$, so it must be a constant $0$ sequence.
\end{proof}

\begin{proposition}\label{cor E}
 Let $Q$ be an acyclic quiver of affine type $\widetilde{\mathbb{E}}$. Then the dimension of the frieze variety $X(Q)$ is equal to one.
\end{proposition}
\begin{proof}
Define the constant $C(n)=X_M|_{x_i=1}$ to be the specialization of the image of $M$ under the Caldero-Chapoton map at $x_1=\cdots=x_{n}=1$, and let 
\[\rho=(C(n)+\sqrt{C(n)^2-4})/2.\]
By definition of $M$, we have $C(n)>2$ and therefore $\rho\ne 1$
 
We checked, by a computer, that for every $i$ and $r$ that the sequence  $(f_i(jm_n+r))_{j=0}^\infty$  satisfies the (not-necessarily minimal) linear recurrence whose  characteristic polynomial is
$$
\prod_{w=-6}^6 (x-\rho^w). 
$$ 
Indeed, by Lemma \ref{lem:0 sequence}, we only need to check that the linear recurrence holds for the first $d_n$ instances, for each $i\in Q_0$, $r=0,\dots,m_n-1$, and each orientation.\footnote{It took a few seconds for $n=7, 8$, and about half an hour for $n=9$ on an iMac.}

The statement then follows from Lemma~\ref{lem dim1}. 
\end{proof}

\subsection{A geometric remark}\label{georem}

Siegel's theorem on integral points  says that a smooth curve of genus at least one has only finitely many integral points. So in the affine case, each component of $X(Q)$ is either of genus zero or singular.  
 We conjecture that each component is a smooth curve of genus 0.

%:

%%%%%%%%%%%%%%%%%%%%%%%%%%%%%%%%%%%%%%%%%%%%%%%%%%%%%%%%%%%%%%
%
%      IT'S
%
%
 %%
%
%      GETTING
%
%
 %%
%
%     WILD
%
%
 %%
\section{Proof of the main theorem part (c), the wild case}\label{sect wild}
This section is divided into two subsections; in the first we recall facts on the Coxeter transformation and in the second  we prove that $\dim X(Q)>1$ for wild type. We keep the notation of the previous sections.

\subsection{Coxeter transformation}
We recall some facts on the Coxeter matrix and its inverse, following the survey paper \cite{delaPena} (rewritten in our notation). %These facts will be important in the proof of our main theorem in the wild case.
\subsubsection{The Coxeter matrix $\Phi$}
Let $C=(c_{ij})_{1\le i,j\le n}$ be the \emph{Cartan matrix} of $Q$, where $c_{ij}$ is the number of paths from $j$ to $i$. Its inverse $C^{-1}$ is the matrix $(b_{ij})_{1\le i,j\le n}$ where $b_{ii}=1$ and if $i\neq j$, then $-b_{ij}$ is the number of arrows from $j$ to $i$ in $Q$. 
Define the \emph{Coxeter matrix} $\Phi$ and its inverse $\Phi^{-1}$ as 
\[\xymatrix{\Phi=-C^T(C^{-1}) &\quad&  \Phi^{-1}=-C(C^{-1})^T}\]
 Then $\Phi^{-1}\dimv M=\dimv (\tau^{-1}M)$ if $M$ is not injective and $\Phi^{-1}\dimv I(i)=-\dimv P(i)$. See for example \cite[\S3.1]{Schiffler}.

Let $\rho_1,\dots,\rho_n$ be the eigenvalues of $\Phi^{-1}$ such that $|\rho_1|\ge|\rho_2|\ge\cdots\ge|\rho_n|$, and $\mathbf{v}_i=[v_{i1}\;\cdots\;v_{in}]^T$ a corresponding generalized eigenvector. The largest absolute value of the eigenvalues $|\rho_1|$ is called the \emph{spectral radius} of $\Phi^{-1}$.

Recall that the characteristic polynomial of a matrix $A$ is defined as $\chi_A(x)=\det(xI-A)$.
Now we recall some properties of the characteristic polynomial $\chi_{\Phi^{-1}}(x)$ (which is called the Coxeter polynomial in \cite{delaPena}).

\begin{lemma}\label{reciprocal} 
{\rm (1)} The following characteristic polynomials are equal:
$$\chi_{\Phi^{-1}}(x)=\chi_{\Phi^T}(x)=\chi_\Phi(x).$$ 
Therefore $\Phi^{-1}$, $\Phi^T$, $\Phi$ have the same set of eigenvalues and the corresponding multiplicities.
Moreover, the polynomials are monic, reciprocal and have integral coefficients; that is,  if we write $\chi_{\Phi^{-1}}(x)=\sum_{i=0}^n a_ix^i$, then   $a_n=1$ and $a_i=a_{n-i}\in\mathbb{Z}$ for all $i$.

In {\rm (2)} and {\rm (3)} we assume that $Q$ is an acyclic wild quiver.

{\rm (2)} The eigenvalue $\rho_1$ is equal to a real number $\rho>1$, and has multiplicity 1. Moreover $|\rho_i|<\rho$ for all $i\neq 1$. 
As a consequence,  $\mathbf{v}_1$ is unique up to scale. Moreover,  we can choose $\mathbf{v}_1\in\mathbb{R}_{>0}^n$, that is, all the coordinates of $\mathbf{v}_1=[v_{11}\;\cdots\; v_{1n}]^T$ are strictly positive.

{\rm (3)} $\rho_n=1/\rho<1$ has multiplicity $1$,  and $|\rho_i|>1/\rho$ for all $i\neq n$. As a consequence, $\mathbf{v}_n$ is unique up to scale. Moreover,  we can choose $\mathbf{v}_n\in\mathbb{R}_{>0}^n$, that is, all the coordinates of $\mathbf{v}_n=[v_{n1}\;\cdots\; v_{nn}]^T$ are strictly positive.

\end{lemma}
\begin{proof}
Most of the lemma is proved in \cite{delaPena}. The notation $M$, $M^{-t}$, $C$ in  \cite{delaPena} correspond to our $(C^{-1})^T$, $C$, $\Phi^T$, respectively. (Below we shall also see that $\rho(C)$, $y^-$, $y^+$ in \cite{delaPena} correspond to our $\rho$, $\mathbf{v}_1^T$, $\mathbf{v}_n^T$.)

(1) Since $$C^{-1}\Phi^{-1}C=C^{-1}(-C(C^{-1})^T)C=-(C^{-1})^TC=\Phi^T,$$ 
we see that $\Phi^{-1}$ and $\Phi^T$ are similar, so $\chi_{\Phi^{-1}}(x)=\chi_{\Phi^T}(x)$. Moreover, a matrix and its transpose have the same characteristic polynomial, so $\chi_{\Phi^T}(x)=\chi_{\Phi}(x)$.   

Moreover, note that  $\det(xI-\Phi^{-1})$  has the leading coefficient $a_n=1$, it has integral coefficients because all entries of $\Phi^{-1}$ are integers (see \S\ref{subsection: basic definition}), and the reciprocal property is proved in \cite[\S2.7]{delaPena}.

(2) 
It is a result by Ringel \cite{R} (Theorem 2.1 in \cite{delaPena}) that $\rho>1$, that it is an eigenvalue of $\Phi^T$ of multiplicity 1, and that other eigenvalues of $\Phi^T$ have norm less than $\rho$. %The first statement of (2) immediate follows.

It is asserted in \cite[\S3.4]{delaPena} that there exists a (row) vector $y^-$ with positive coordinates such that $y^-\Phi^T=\rho^{-1} y^+$. Thus $\Phi (y^-)^T=\rho^{-1}(y^-)^T$, therefore $\Phi^{-1}(y^-)^T=\rho (y^-)^T$. So we can take $\mathbf{v}_1=(y^-)^T$. This proves the last statement of (2).

(3) The first statement of (3) follows from (1) and (2); indeed, because $\chi_{\Phi^{-1}}(x)$ being reciprocal is equivalent to  $\chi_{\Phi^{-1}}(x)=x^n \chi_{\Phi^{-1}}(x^{-1})$  \cite[\S2.7]{delaPena}, we have that $\rho_i$ and $\rho_i^{-1}$ are eigenvalues with the same multiplicity for every $1\le i\le n$. 

The proof of the second statement of (3) is similar to the proof of the second statement of (2), where $\mathbf{v}_n=(y^+)^T$ for the vector $y^+$ defined in \cite[\S3.4]{delaPena}.
\end{proof}

\begin{lemma}\label{rho is irrational} The eigenvalue
$\rho$ is irrational.
\end{lemma}
\begin{proof}
The eigenvalue $\rho$ is a root of the characteristic polynomial $\chi_{\Phi^{-1}}(x)$ which, by Lemma~\ref{reciprocal}, is an integral-coefficient polynomial whose leading coefficient and constant are both 1. The only possible rational roots  of this polynomial are $\pm1$ by the rational root theorem. But $\rho>1$, so $\rho$ must be irrational. 
\end{proof}

In the rest of the paper we require that, for $i=1,n$,
$$\mathbf{v}_i\in\mathbb{R}^n_{>0}\textrm{ and } ||\mathbf{v}_i||=1.$$ 
Such $\mathbf{v}_1$ and $\mathbf{v}_n$ exist uniquely by the above lemma.
\medskip

\begin{lemma}\label{dim asymptotic} Let $Q$ be a wild acyclic quiver and $M$ an indecomposable, preprojective representation. Then
 $\displaystyle\lim_{t\to \infty}\frac{1}{\rho^t}\, \dimv \tau^{-t}M=\lambda \mathbf{v}_1$ for some real number $\lambda>0$.
As a consequence, there exist  $c, N\in\mathbb{R}_{>0}$ such that, all components of $\dimv \tau^{-t}P(i)$ are greater than or equal to $c \rho^t$ for every $t\ge N$ and every $i\in Q_0$.
\end{lemma}
\begin{proof}
The first statement is \cite[Theorem 3.5]{delaPena}. The consequence is obvious.
\end{proof}

Indeed, a weaker version of Lemma \ref{dim asymptotic} (replacing $\lambda>0$ by $\lambda\ge0$) is easy to prove, as shown in (1) of the lemma below. 

Recall that the {\em norm of a matrix} $A$ is defined as 
$$||A||:=\sup_{||x||=1} ||Ax||$$ and it satisfies  the following inequalities, see for example \cite[Theorem 14 on page 90]{Lax}.
\begin{equation}\label{eq24} ||A+B||\le ||A||+||B||  \quad \textup{and} \quad ||A\,B|| \le ||A|| \,||B||.
\end{equation}
\begin{lemma}\label{limit of Phi^-t}
(1) For any vector $\mathbf{v}\in\mathbb{R}^n$, there exists $\lambda\in\mathbb{R}$ such that $\displaystyle\lim_{t\to \infty}\frac{1}{\rho^t}\, \Phi^{-t}\mathbf{v}=\lambda \mathbf{v}_1$.

(2) There exists a number  $N$ such that $||\frac{1}{\rho^t}\, \Phi^{-t}||\le N$ for every $t\in\mathbb{Z}_{\ge0}$.
\end{lemma}
\begin{proof}
(1) Let $\mathbf{v}_1,\dots,\mathbf{v}_n$ be generalized eigenvectors  corresponding to eigenvalues $\rho_1,\rho_2,\dots,\rho_n$ such that
$$\Phi^{-1}\mathbf{V}=\mathbf{V}J,\quad \textrm{(denote $\mathbf{V}:=\begin{bmatrix}\mathbf{v}_1&\cdots&\mathbf{v}_n\end{bmatrix}$)}$$
where $J$ is the Jordan normal form, thus
$$J=\begin{bmatrix}\rho&0&\cdots &0\\ 0&J_2&\cdots&0\\ \vdots\\ 0&0& \cdots & J_p\end{bmatrix},\textrm{ where each $J_j$ is of the form }
\begin{bmatrix}\rho_i &1&0&\cdots&0\\0&\rho_i&1&\cdots&0\\ \vdots&&&&1\\ 0&0&0&\cdots& \rho_i\end{bmatrix}\textrm{ for some $i\neq 1$}.$$
Denote the size of $J_j$ as $m_j\times m_j$. We decompose the Jordan blocks
$$J_j=\rho_i I+K, \textrm{ where }
I=\begin{bmatrix}1&0&0&\cdots&0\\0&1&0&\cdots&0\\ \vdots&&&&0\\ 0&0&0&\cdots& 1\end{bmatrix},
K=\begin{bmatrix}0 &1&0&\cdots&0\\0&0&1&\cdots&0\\ \vdots&&&&1\\ 0&0&0&\cdots& 0\end{bmatrix}
$$
Since $IK=KI=K$ and $K^{\ell}=0$ for $\ell\ge m_j$, we have
$$\frac{1}{\rho^t}J_j^t=\frac{1}{\rho^t}(\rho_i I+K)^t=\sum_{\ell=0}^{m_j-1} \frac{\rho_i^{t-\ell}}{\rho^t}\binom{t}{\ell}K^\ell$$
Since $|\rho_i/\rho|<1$, and the exponential grows faster than a polynomial, we have
$$\lim_{t\to\infty} \Big|\frac{\rho_i^{t-\ell}}{\rho^t}\binom{t}{\ell}\Big|
=|\rho_i^{-\ell}|
\lim_{t\to\infty} \left|\frac{\rho_i}{\rho}\right|^t\binom{t}{\ell}=0,\quad \textrm{ for each $\ell=0,\dots,m_j-1$.}
$$
Therefore every entry of the matrix $\frac{1}{\rho^t}J_j^t$ approaches 0 as $t\to\infty$. Thus
\begin{equation}\label{E11}
\lim_{t\to\infty}\frac{1}{\rho^t}J^t
=\begin{bmatrix}1&0&\cdots&0\\ 0&0&\cdots&0\\ \vdots&&\vdots\\ 0&0&\cdots& 0\end{bmatrix}=:E_{11}.
\end{equation}
Writing $\mathbf{v}=\sum_i\lambda_i\mathbf{v}_i=\mathbf{V}\Lambda$ (where $\Lambda=\begin{bmatrix}\lambda_1&\cdots&\lambda_n\end{bmatrix}^T$), and noticing that $\Phi^{-t}\mathbf{V}=\mathbf{V}J^t$, we conclude
$$\lim_{t\to \infty}\frac{1}{\rho^t}\, \Phi^{-t}\mathbf{v}
=\lim_{t\to \infty}\frac{1}{\rho^t}\, \Phi^{-t}\mathbf{V}\Lambda
=\lim_{t\to \infty}\frac{1}{\rho^t}\, \mathbf{V}J^t\Lambda
=\mathbf{V}(\lim_{t\to \infty}\frac{1}{\rho^t}\, J^t)\Lambda
=\mathbf{V}E_{11}\Lambda
=\lambda_1 \mathbf{v}_1.$$

(2) It follows from \eqref{E11} that $||\frac{1}{\rho^t}J^t||\to 1$ as $t\to\infty$, so there exists a number $N'$ such that
$$||\frac{1}{\rho^t}J^t||\le N', \quad \textrm{ for every $t\in\mathbb{Z}_{\ge0}$.}$$ 
Then using the inequality  $||AB||\le ||A||\cdot ||B||$ , we have
$$||\frac{1}{\rho^t}\Phi^{-t}||  \le ||\mathbf{V}|| \cdot ||\frac{1}{\rho^t}J^t||\cdot ||\mathbf{V}^{-1}|| \le 
||\mathbf{V}|| \cdot N'\cdot ||\mathbf{V}^{-1}|| =: N, \quad \textrm{ for every $t\in\mathbb{Z}_{\ge0}$.}
$$
\end{proof}
Similarly, we have
\begin{lemma}\label{limit of Phi^t}
(1) For any vector $\mathbf{v}\in\mathbb{R}^n$, there exists $\lambda\in\mathbb{R}$ such that $\displaystyle\lim_{t\to \infty}\frac{1}{\rho^t}\, \Phi^{t}\mathbf{v}=\lambda \mathbf{v}_n$.

(2) There exists a number  $N$ such that $||\frac{1}{\rho^t}\, \Phi^{t}||\le N$ for every $t\in\mathbb{Z}_{\ge0}$.
\end{lemma}
\begin{proof}
Note that since $\rho_1,\dots,\rho_n$ are eigenvalues of $\Phi^{-1}$, we see that $\rho_1^{-1},\dots,\rho_n^{-1}$ are eigenvalues of $\Phi$, and that $\rho_1^{-1}=1/\rho$ and $\rho_n^{-1}=\rho$ have the largest and smallest norm. Moreover, since
$\Phi^{-1}\mathbf{v}_n=\rho^{-1} \mathbf{v}_n$, we have
$\Phi\mathbf{v}_n=\rho \mathbf{v}_n$, that is, $\mathbf{v}_n$ is an eigenvector of $\Phi$ corresponding to the eigenvalue $\rho$. Then this lemma follows from Lemma \ref{limit of Phi^-t}.
\end{proof}

\subsection{Proof of Theorem \ref{thm main} (c)}
We first need two results on the growth of the coefficients $f_i(t)$ in terms of the spectral radius $\rho$ of the inverse Coxeter matrix $\Phi^{-1}$.

\begin{lemma}\label{f_i(t) greater than}
Let $d(t)$ be the largest coordinate in the vector $\dimv \tau^{-t}P(i)$. Then $f_i(t)\ge2^{d(t)}$. As a consequence, there exist  $c, N_1\in\mathbb{R}_{>0}$ such that $f_i(t)\ge 2^{c \rho^t}$ for every $t\ge N_1$, $i\in Q_0$.
\end{lemma}
\begin{proof}
Assume $(d_1(t),\dots,d_n(t))=\dimv \tau^{-t}P(i)$ and $d(t)=d_j(t)$ is the largest coordinate. Let 
$$X=X_{\tau^{-t}P(i)}\ =\ \sum_{r\in\mathbb{Z}} c_rx_j^r 
\ =\  c_{-d(t)}x_j^{-d(t)} +\sum_{r>-d(t)} c_rx_j^r ,$$
%,\quad \textrm{ where } c_r\in R_{\widehat{j}}:=\mathbb{Z}_{\ge0}[x_1^{\pm},\dots,x_{j-1}^{\pm},x_{j+1}^{\pm},\dots,x_n^{\pm}]$$ 
 where $ c_r\in R_{\widehat{j}}:=\mathbb{Z}_{\ge0}[x_1^{\pm},\dots,x_{j-1}^{\pm},x_{j+1}^{\pm},\dots,x_n^{\pm}]$,  
be the Laurent expansion of the cluster variable corresponding to $\tau^{-t}P(i)$ in the initial cluster $\mathbf{x}_0$. Note that $ c_{-d(t)}\ne0$ because of Lemma \ref{lem 123}(3). Let $x_j'$ denote the cluster variable obtained by mutating $x_j$ of the initial cluster at $j$; that is,    
$x_j'=(P+Q)/x_j$ where $P=\prod_{k\to j}x_k$ and $Q=\prod_{k\leftarrow j}x_k$ are both in $R_{\widehat{j}}$.
Because of the Laurent phenomenon \cite{FZ} and the positivity theorem \cite{LS4}, the cluster variable $X$ is a $\mathbb{Z}_{\ge0}$-coefficient Laurent polynomial with respect to any initial cluster. Particularly, taking the initial cluster to be $\{x_1,\dots,x_j',\dots,x_n\}$, the cluster variable
$$X=\sum_{r\in\mathbb{Z}} c_rx_j^r
=\sum_{r\in\mathbb{Z}} c_r \left(\frac{P+Q}{x'_j}\right)^r
=\sum_{r\in\mathbb{Z}} c_r(P+Q)^r(x'_j)^{-r}$$
must actually be an element in 
$$\mathbb{Z}_{\ge0}[x_1^{\pm},\dots,(x_j')^{\pm},\dots, x_n^{\pm}]=R_{\widehat{j}}[(x_j')^{\pm}]$$
(Note that, $X$ is a priori only an element in $K_{\widehat{j}}[(x'_j)^\pm]$, where $K_{\widehat{j}}$ is the field of rational functions
$K_{\widehat{j}}:=\mathbb{Q}(x_1,\dots,x_{j-1},x_{j+1},\dots,x_n)$.\ )
Viewing $X$ as a Laurent polynomial in $x'_j$ with coefficient in $R_{\widehat{j}}$, we see that for each $(x_j')^r$, its coefficient $c_r(P+Q)^r$ must be in  $R_{\widehat{j}}$. In particular, $c_{-d(t)}(P+Q)^{-d(t)}$ must be in $R_{\widehat{j}}$. 
Thus $(P+Q)^{d(t)}$ divides $c_{-d(t)}$.
%Since $c_{-d(t)}\neq0$ by the definition of denominator vectors (\cite[page 135]{Cluster IV}) and $ (P+Q)^{d(t)}| c_{-d(t)}$, 
Specializing initial cluster variables at 1, we have $c_{-d(t)}|_{x_1=\cdots=x_n=1}\ge 2^{d(t)}$. Therefore $f_i(t)=X|_{x_1=\cdots=x_n=1}\ge 2^{d(t)}$. This proves the first statement.

The consequence follows from Lemma \ref{dim asymptotic}.
\end{proof}

We now consider the natural logarithm of the integers $f_i(t)$.
Denote $L_i(t)=\ln f_i(t)\in\mathbb{R}$ for $i\in Q_0$, $t\in\mathbb{Z}_{\ge0}$. For each $t\in\mathbb{Z}_{\ge0}$, define a column vector
$$\mathbf{L}(t)=\begin{bmatrix}L_1(t)\\ \vdots \\ L_n(t)\end{bmatrix}\in\mathbb{R}^n.
$$
Recall that $\mathbf{v}_n$ is defined in Lemma \ref{reciprocal}.
\begin{proposition}\label{limit L(t)} 
$$\lim_{t\to\infty}\frac{1}{\rho^t}\mathbf{L}(t)=\eta\mathbf{v}_n$$ for some real number $\eta>0$.
\end{proposition}
\begin{proof}
We first prove that the equality holds for some real number $\eta$. The idea is to show that, for $s$ sufficiently large, the growth of $\mathbf{L}(s+t)$ and
$\Phi^{-t}\mathbf{L}(s)$ are almost the same as $t\to\infty$, and the latter is well understood by Lemma \ref{limit of Phi^-t}.

Rewrite \eqref{def:fi} as
$$f_i(t+1)f_i(t)=1+\prod_{j\to i}f_j(t) \prod_{j\leftarrow i}f_j(t+1)$$
Taking the logarithm on both sides and using the fact\footnote{To see $\ln(x+1)-\ln x<1/x$, note the left side is $\ln(x+1)-\ln x=\ln (1+1/x)$. Replacing $1/x$ by $x$, it is equivalent to show $f(x)=x-\ln (1+x)>0$ for all $x> 0$. Note that $f(0)=0$. Since $f'(x)=1-1/(1+x)>0$ for $x>0$, $f(x)$ is strictly increasing for $x\ge 0$, so the inequality follows.} 
that $0< \ln(x+1)-\ln x < 1/x$ for any positive real number $x$,
we conclude
$$L_i(t+1)+L_i(t)=\sum_{j\to i}L_j(t) + \sum_{j\leftarrow i}L_j(t+1)+\delta_i(t),$$
where $\delta_i(t)>0$, and by Lemma \ref{f_i(t) greater than},
\begin{equation}\label{delta_i(t) less than}
\delta_i(t)<\frac{1}{\prod_{j\to i}f_j(t) \prod_{j\leftarrow i}f_j(t+1)}<2^{-c\rho^t}, \textrm{ for $t\ge N_1$}.
\end{equation}

Then
$$L_i(t+1)-\sum_{j\leftarrow i}L_j(t+1)=-L_i(t) +\sum_{j\to i}L_j(t) +\delta_i(t).$$
Rewriting these equations in matrix form brings up the inverse Cartan matrix as follows.
$$
\begin{bmatrix}1&0&\cdots&&0\\-b_{12}&1&0&\cdots&0\\-b_{13}&-b_{23}&1&\cdots&0\\ \vdots&&&&\vdots\\ -b_{1n}&-b_{2n}&-b_{3n}&\cdots&1\end{bmatrix}\mathbf{L}(t+1)
=
\begin{bmatrix}-1&b_{12}&\cdots&&b_{1n}\\ 0&-1&b_{23}&\cdots&b_{2n}\\0&0&-1&\cdots&b_{3n}\\ \vdots&&&&\vdots\\ 0&0&0&\cdots&-1\end{bmatrix}\mathbf{L}(t)
+
\begin{bmatrix}
\delta_1(t)\\\delta_2(t)\\\delta_3(t)\\ \vdots\\\delta_n(t)
\end{bmatrix}
$$
which is
$$
(C^{-1})^T\mathbf{L}(t+1)
=
-(C^{-1})\mathbf{L}(t)
+\boldsymbol{\delta}(t)
$$
where $\boldsymbol{\delta}(t)=\begin{bmatrix}\delta_1(t)&\cdots&\delta_n(t)\end{bmatrix}^T$ satisfying (by \eqref{delta_i(t) less than})
\begin{equation}\label{delta(t) less than}
||\boldsymbol{\delta}(t)|| <  2^{-c\rho^t}\sqrt{n}\; <\sqrt{n}, \quad\textrm{ for $t\ge N_1$.}
\end{equation}
Left-multiplying the above equality by $C^T$, we obtain
$$
\mathbf{L}(t+1)
=
\Phi\mathbf{L}(t)
+C\boldsymbol{\delta}(t)
$$
It follows that, for $s,t\in\mathbb{Z}_{\ge0}$,
$$
\mathbf{L}(s+t)
=
\Phi^{t}\mathbf{L}(s)+\Phi^{t-1}C\boldsymbol{\delta}(s)+\Phi^{t-2}C\boldsymbol{\delta}(s+1)+\cdots+C\boldsymbol{\delta}(s+t-1)
$$
Assume $s\ge N_1$. Then
$$\aligned
&\left\Vert\frac{1}{\rho^{s+t}}
\Big(\mathbf{L}(s+t)
-
\Phi^{t}\mathbf{L}(s)\Big)\right\Vert
=\frac{1}{\rho^{s+t}}\Vert\sum_{i=1}^{t}\Phi^{t-i} C\boldsymbol{\delta}(s+i-1)\Vert\\
&\le\sum_{i=1}^{t}\frac{1}{\rho^{s+i}} \Vert \frac{1}{\rho^{t-i}} \Phi^{t-i}\Vert\cdot \Vert C\boldsymbol{\delta}(s+i-1)\Vert \quad\textrm{by \eqref{eq24}}\\
&\le\sum_{i=1}^{t}\frac{1}{\rho^{s+i}} N \Vert C\Vert \cdot \Vert\boldsymbol{\delta}(s+i-1)\Vert
\quad\textrm{ (This $N$ is the constant in Lemma \ref{limit of Phi^t} (2)) }\\
&
<\sum_{i=1}^{t}\frac{1}{\rho^{s+i}} N \Vert C\Vert \cdot \sqrt{n}\quad \textrm{(by \eqref{delta(t) less than} and  $s\ge N_1$)} \\
&< N \Vert C\Vert \sqrt{n}\sum_{i=1}^{\infty}\frac{1}{\rho^{s+i}}
= N \Vert C\Vert \sqrt{n}\,\frac{\rho^{-s-1}}{1-\rho^{-1}}\to 0 \textrm{ as $s\to\infty$.}
\endaligned
$$
Therefore for any $\epsilon>0$, we can find an integer $s_0$ sufficiently large such that
\begin{equation}\label{s0}
\left\Vert\frac{1}{\rho^{s_0+t}}
\Big(\mathbf{L}(s_0+t)
-
\Phi^{t}\mathbf{L}(s_0)\Big)\right\Vert<\epsilon,\quad \forall t\ge0.
\end{equation}
By Lemma \ref{limit of Phi^t} (1), there exists $N_2\in\mathbb{R}_{>0}$ and $\lambda\in\mathbb{R}$ such that 
\begin{equation}\label{N3}
\left\Vert\frac{1}{\rho^t}\, \Phi^{t}\Big(\rho^{-s_0}\mathbf{L}(s_0)\Big)-\lambda \mathbf{v}_n\right\Vert<\epsilon, \quad \forall t\ge N_2
\end{equation}
Adding the inequalities \eqref{s0} and \eqref{N3} and using the triangle inequality \eqref{eq24}, we have
$$\left\Vert\frac{1}{\rho^{s_0+t}}
\mathbf{L}(s_0+t)
-
\lambda\mathbf{v}_n\right\Vert<2\epsilon,\quad \forall t\ge N_2
$$
Now replace $s_0+t$ by $t$. 
We have that, for any $\epsilon>0$, there exist $\lambda\in\mathbb{R}$, $N_4\in\mathbb{R}_{>0}$ (we can just take $N_4=s_0+N_2$), such that 
\begin{equation}\label{N4}
\left\Vert
\frac{1}{\rho^{t}}\mathbf{L}(t)-\lambda\mathbf{v}_n
\right\Vert<2\epsilon, \quad \forall t \ge N_4
\end{equation}
thus for any $\epsilon>0$, there exists $N_4\in\mathbb{R}_{>0}$ such that 
$$
\left\Vert
\frac{1}{\rho^{t}}\mathbf{L}(t)-\frac{1}{\rho^{t'}}\mathbf{L}(t')
\right\Vert<4\epsilon, \quad \forall t,t' \ge N_4
$$
Therefore 
$\{\frac{1}{\rho^{t}}\mathbf{L}(t)\}$
is a Cauchy sequence, so must converge. Denote $\mathbf{u}:=\lim_{t\to\infty}\frac{1}{\rho^{t}}\mathbf{L}(t)$. If $\mathbf{u}$ is not on the line spanned by $\mathbf{v}_n$, choose $\epsilon>0$ such that $3\epsilon$ is less than the distance from $\mathbf{u}$ and that line. Taking $t\to\infty$ in \eqref{N4} gives the contradiction
$$3\epsilon<||\mathbf{u}-\lambda\mathbf{v}_n||<2\epsilon.$$
Therefore $\mathbf{u}=\eta\mathbf{v}_n$ for some $\eta\in\mathbb{R}$.

Finally, we show that $\eta>0$. 
By Lemma \ref{f_i(t) greater than}, there exists  $c, N_1\in\mathbb{R}_{>0}$ such that 
$$L_i(t)=\ln f_i(t)\ge {c \rho^t}\ln 2,\quad \textrm{ for every $t\ge N_1$ and every $i\in Q_0$.}$$
So
$$\eta \mathbf{v}_n=\lim_{t\to\infty}\frac{1}{\rho^t}\mathbf{L}(t)\ge [c\ln 2\;\; c\ln 2\;\; \cdots \;\; c\ln2]^T$$ (here we mean that ``$\ge$'' holds componentwise), which implies $\eta>0$.
\end{proof}

We also need the following simple result.
\begin{lemma}\label{lem ln}
For all $a_1,\dots,a_m,x_1,\dots,x_m>0$
$$\ln(a_1x_1+\cdots +a_mx_m) \ \le\  \max(\ln (x_i))+\max(\ln a_i)+\ln m.$$
\end{lemma}
\begin{proof}
The left hand side is at most $ \ln(m\max(a_i)\max(x_i))$ which is equal to the right hand side.
\end{proof}

We are now ready for the proof of our main result.
\begin{proof}[Proof of Theorem \ref{thm main}(c)]
Suppose $X(Q)$ is contained in a 1-dimensional variety.

Consider the projection $\pi: \mathbb{C}^n\to \mathbb{C}^2$, $(x_1,\dots,x_n)\mapsto(x_1,x_2)$. Then $\pi(X(Q))$ is at most 1-dimensional. So there exists a nonzero polynomial $g(x,y)=\sum_{(i,j)\in S} a_{ij}x^iy^j\in\mathbb{C}[x,y]$ (where $a_{ij}\neq0$ for every $(i,j)\in S\subset \mathbb{Z}_{\ge0}^2$) such that $g(f_1(t),f_2(t))=0$ for every $t$.

For convenience, denote the $i$-th coordinates $v_{ni}$ of  the eigenvector $\mathbf{v}_n$ by $y_i$ for $i=1,\dots,n$, that is, $\mathbf{v}_n=[y_1\;\cdots\;y_n]^T$.  Let $(i_0,j_0)\in S$ such that $iy_1+jy_2$ is maximal. Replacing $g$ by $g/a_{i_0j_0}$ if necessary, we may assume $a_{i_0j_0}=1$. Then
$$f_1(t)^{i_0}f_2(t)^{j_0}=\sum_{(i,j)\in S\setminus(i_0,j_0)} (-a_{ij})f_1(t)^i f_2(t)^j$$
Taking the logarithm on both sides, we get
$$i_0 L_1(t)+j_0L_2(t)=\ln \Big(\sum_{(i,j)\in S\setminus(i_0,j_0)} (-a_{ij})f_1(t)^i f_2(t)^j\Big)$$
and according to Lemma \ref{lem ln}
we get 
$$i_0 L_1(t)+j_0L_2(t) \le \max_{(i,j)\in S\setminus(i_0,j_0)}\big(iL_1(t)+jL_2(t)\big)+\max_{(i,j)\in S\setminus(i_0,j_0)}(\ln|a_{ij}|)+\ln(|S|-1).$$
Now we use Proposition \ref{limit L(t)}. 
Dividing the above inequality by $\rho^t$ and letting $t\to\infty$, we conclude that there exists $(i,j)\in S\setminus(i_0,j_0)$ such that
$$i_0\eta y_1+j_0\eta y_2\le i\eta y_1+j\eta y_2.$$
Thus
$$i_0 y_1+j_0 y_2\le i y_1+j y_2.$$
By the choice of $(i_0,j_0)$, the equality must hold. Therefore $(i_0-i)(y_1/y_2)=(j-j_0)$. Since $y_1/y_2\neq0$  and $(i,j)\neq(i_0,j_0)$, we must have $j-j_0$ and $i_0-i$ both been nonzero. Thus
 $y_1/y_2=(j-j_0)/(i_0-i)$ is rational. 

By a similar argument, $y_i/y_j$ is rational for any $1\le i<j\le n$. So there is a constant $c\in\mathbb{R}_{>0}$ such that $c\mathbf{v}_n\in\mathbb{Q}_{>0}^n$. Then it follows from $$\rho(c\mathbf{v}_n)=\Phi(c\mathbf{v}_n)\in \mathbb{Q}^n$$ that $\rho$ is rational, which contradicts Lemma \ref{rho is irrational}. 

This completes the proof.
\end{proof}

\begin{remark}
It is natural to ask what is the exact dimension of $X(Q)$ in the wild type. One might be tempted to conjecture that it is always equal to the number of vertices $n$. However, this is not true, because whenever the  quiver has a non-trivial automorphism $\phi\in \Aut(Q)$ then we have $f_i(t)=f_{\phi(i)}(t)$, for all $t$, so the dimension cannot be $n$. 

An upper bound for the dimension is the number of orbits under the action of $\Aut(Q)$ on $Q_0$.
\end{remark}
\section{Examples}\label{sect 4}

\subsection{Type $\mathbb{A}_{1,1}$}
For the Kronecker quiver $\xymatrix{1&\ar@<2pt>[l]\ar@<-2pt>[l]2}$ the points $P_t=(f_1(t),f_2(t))$ are $(1,1),(2,5), (13,34), (89, 233) , \ldots$ whose coordinates consist of every other Fibonacci number. The frieze variety $X(Q) \subset \mathbb{C}^2$ is given by the polynomial $x^2-3xy+y^2+1$. This is a smooth curve of genus zero.  
\subsection{Type $\mathbb{A}_{2,1}$}
In this case the first few points are \[(1,1,1), (2,3,7), (11,26,41), (97,153,362),(571,1351,2131),\ldots\] The frieze variety $X(Q)$ has two components, and each is a planar curve of degree 2.
$ V(x_1-2x_2+x_3, 2x_2^2-6x_2x_3+3x_3^2+1)$ and 
 $V(x_1-3x_2+x_3, 3x_2^2-6x_2x_3+2x_3^2+1)$. Both are smooth curves of degree 2, hence of genus zero.

\subsection{Type $\mathbb{D}_5$}
Consider the following quiver $Q$:
\[\xymatrix@!@R=-28pt@C=15pt{1 &&& 5 \ar@{->}[ld]\\
&3\ar@{->}[dl]\ar@{->}[ul] & 4\ar@{->}[l]  \\  
2  &&& 6 \ar@{->}[lu] \\ }\]
The first five points are $$(1,1,1,1,1,1),(2,2,5,6,7,7),(3,3,11,90,13,13),(4,4,131,246,19,19), (33,33,2045,3001,158,158).$$

The frieze variety $X(Q)$ has three components. 
$$ V(x_5-x_6, x_1-x_2,x_6^2 + x_3 - 2x_4, 5x_2x_6 - x_3 - x_4 -3, x_2^2 -2x_3 + x_4), $$
$$
V(x_5-x_6, x_1-x_2,x_6^2 + x_3 - 9x_4, 5x_2x_6 - x_3 -8 x_4 -17, x_2^2 -2x_3 + x_4) , $$
$$
V(x_5-x_6, x_1-x_2,x_6^2 + x_3 - 2x_4, 5x_2x_6 - 8x_3 - x_4 -17, x_2^2 -9x_3 + x_4).$$
Each component is a smooth curve of degree 2 and genus zero by the same argument as in the previous example.
\subsection{A wild example}
Consider the following quiver $Q$:
%\[
%\begin{tikzcd}[arrow style=tikz,>=stealth,row sep=4em]
%& 2 \arrow[dl]\\
%1&& 3\arrow[ll,shift left=.4ex] \arrow[ll,shift right=.4ex]\arrow[ul]\\
%\end{tikzcd}
%\]
\[\xymatrix{&2\ar[ld]\\ 1&&3\ar@<2pt>[ll]\ar@<-2pt>[ll]\ar[lu]}\]
Then
$$\begin{tabu}{|l|lll|ccc|}\hline
  t & f_1(t)&f_2(t)&f_3(t) & &\mathbf{L}(t)^T& \\
  \hline
  1 & 2 & 3& 13&0.693& 1.099& 2.565\\
  2 & 254 &1101& 5464009&5.537&7.004& 15.514\\
  3&1.294\times10^{14}&6.422\times10^{17}&1.969\times10^{39}&32.49& 41.00& 90.48\\
  4&1.923\times10^{82}&5.895\times10^{103}&1.107\times10^{229}&189.47&238.94& 527.39\\
  5&3.759\times10^{479}&7.063\times10^{604}&9.012\times10^{1334}&1104.26&1392.72& 3073.85\\
  \hline
\end{tabu}
$$
The Cartan matrix and its inverse, the Coxeter matrix and its inverse are:
$$
C=\begin{bmatrix}1&1&3\\ 0&1&1\\0&0&1\end{bmatrix},\quad
C^{-1}=\begin{bmatrix}1&-1&-2\\ 0&1&-1\\0&0&1\end{bmatrix}, \quad
\Phi=\begin{bmatrix}-1&1&2\\-1&0&3\\-3&2&6\end{bmatrix} ,\quad
\Phi^{-1}=\begin{bmatrix}6&2&-3\\3&0&-1\\2&1&-1\end{bmatrix} ,
$$
The characteristic polynomial $\chi_{\Phi^{-1}}(x)=x^3-5x^2-5x+1=(x+1)(x^2-6x+1)$, so $\rho=3+\sqrt{8}\approx 5.8284$ and $1/\rho=3-\sqrt{8}\approx 0.1716$ are irrational, and the corresponding eigenvectors are 
$$\mathbf{v}_1\approx[0.866\;\; 0.392\;\; 0.311]^T,\quad 
\mathbf{v}_n\approx[0.311\;\; 0.392\;\; 0.866]^T
$$ 
(In this example, $\mathbf{v}_n$ happens to be the ``reverse'' of $\mathbf{v}_1$. This is not the case in general.)
Computation shows
$$\lim_{t\to\infty}\frac{1}{\rho^t}\mathbf{L}(t)\approx\begin{bmatrix}0.164\\0.207\\0.457\end{bmatrix}\approx 0.528\mathbf{v}_n.$$
So roughly we can describe the growth of $P_i(t)$ as
$$P_i(t)=(f_1(t),f_2(t),f_3(t))\approx (e^{0.164\rho^t},e^{0.207\rho^t},e^{0.457\rho^t}).$$

\end{document}

%% file: figdehnD2.pdf_tex
%% Creator: Inkscape inkscape 0.48.5, www.inkscape.org
%% Accompanies image file 'figdehnD2.pdf' (pdf, eps, ps)
\begingroup%
  \makeatletter%
  \providecommand\color[2][]{%
    \errmessage{(Inkscape) Color is used for the text in Inkscape, but the package 'color.sty' is not loaded}%
    \renewcommand\color[2][]{}%
  }%
  \providecommand\transparent[1]{%
    \errmessage{(Inkscape) Transparency is used (non-zero) for the text in Inkscape, but the package 'transparent.sty' is not loaded}%
    \renewcommand\transparent[1]{}%
  }%
  \providecommand\rotatebox[2]{#2}%
  \ifx\svgwidth\undefined%
    \setlength{\unitlength}{563.575bp}%
    \ifx\svgscale\undefined%
      \relax%
    \else%
      \setlength{\unitlength}{\unitlength * \real{\svgscale}}%
    \fi%
  \else%
    \setlength{\unitlength}{\svgwidth}%
  \fi%
  \global\let\svgwidth\undefined%
  \global\let\svgscale\undefined%
  \makeatother%
  \begin{picture}(1,0.44791306)%
    \put(0,0){\includegraphics[width=\unitlength]{figdehnD2.pdf}}%
    \put(0.23706961,0.37349112){\color[rgb]{0,0,0}\makebox(0,0)[lb]{\smash{$\zg$}}}%
    \put(0.18473952,0.3534252){\color[rgb]{0,0,0}\makebox(0,0)[lb]{\smash{$L$}}}%
    \put(0.6093653,0.04065113){\color[rgb]{0,0,0}\makebox(0,0)[lb]{\smash{$\zg_1$}}}%
    \put(0.92449638,0.04065113){\color[rgb]{0,0,0}\makebox(0,0)[lb]{\smash{$\zg_2$}}}%
    \put(0.09871357,0.03305116){\color[rgb]{0,0,0}\makebox(0,0)[lb]{\smash{$\cald^{1/2}(\zg)$}}}%
    \put(0.36558132,0.03305116){\color[rgb]{0,0,0}\makebox(0,0)[lb]{\smash{$\cald^{-1/2}(\zg)$}}}%
  \end{picture}%
\endgroup%